\documentclass[english]{amsart}
\usepackage[T1]{fontenc}
\usepackage[latin9]{inputenc}
\usepackage{geometry}
\usepackage{color}
\usepackage{units}
\usepackage{amstext}
\usepackage{amsthm}
\usepackage{amssymb}
\usepackage{graphicx}
\usepackage[all]{xy}

\makeatletter
\numberwithin{equation}{section}
\numberwithin{figure}{section}
\theoremstyle{plain}
\newtheorem{thm}{\protect\theoremname}[section]
  \theoremstyle{remark}
  \newtheorem{rem}[thm]{\protect\remarkname}
  \theoremstyle{plain}
  \newtheorem{prop}[thm]{\protect\propositionname}
  \theoremstyle{definition}
  \newtheorem{defn}[thm]{\protect\definitionname}
  \theoremstyle{plain}
  \newtheorem{cor}[thm]{\protect\corollaryname}
  \theoremstyle{plain}
  \newtheorem{lem}[thm]{\protect\lemmaname}

\makeatother

\usepackage{babel}
  \providecommand{\corollaryname}{Corollary}
  \providecommand{\definitionname}{Definition}
  \providecommand{\lemmaname}{Lemma}
  \providecommand{\propositionname}{Proposition}
  \providecommand{\remarkname}{Remark}
\providecommand{\theoremname}{Theorem}

\begin{document}
\global\long\def\bbN{\mathbb{N}}

\global\long\def\bbQ{\mathbb{Q}}

\global\long\def\bbR{\mathbb{R}}

\global\long\def\bbZ{\mathbb{Z}}

\global\long\def\bbC{\mathbb{C}}

\global\long\def\eset{\emptyset}

\global\long\def\nto{\nrightarrow}

\global\long\def\re{\mathrm{Re\,}}

\global\long\def\im{\mathrm{Im\,}}

\global\long\def\limti{{\displaystyle \lim_{n\to\infty}}}

\global\long\def\sumnti{{\displaystyle \sum_{n=1}^{\infty}}}

\global\long\def\sumktn{{\displaystyle \sum_{k=1}^{n}}}

\global\long\def\E{{\bf E}}

\global\long\def\ind{{\bf 1}}

\global\long\def\O{O}

\global\long\def\Leb{}

\title[A temporal CLT for cocycles over rotations]{A temporal Central Limit Theorem for real-valued cocycles over rotations}

\author{Michael Bromberg and Corinna Ulcigrai}
\begin{abstract}
We consider deterministic random walks on the real line driven by
irrational rotations, or equivalently, skew product extensions of
a rotation by $\alpha$ where the skewing cocycle is a piecewise constant
mean zero function with a jump by one at a point $\beta$. When $\alpha$
is badly approximable and\textbf{ $\beta$} is badly approximable
with respect to $\alpha$, we prove a \emph{Temporal Central Limit
theorem }(in the terminology recently introduced by D.Dolgopyat and
O.Sarig), namely we show that for any fixed initial point, the \emph{occupancy
random variables}, suitably rescaled, converge to a Gaussian random
variable. This result generalizes and extends a theorem by J. Beck
for the special case when $\alpha$ is quadratic irrational, $\beta$
is rational and the initial point is the origin, recently reproved
and then generalized to cover any initial point using geometric renormalization
arguments by Avila-Dolgopyat-Duryev-Sarig (Israel J., 2015) and Dolgopyat-Sarig
(J. Stat. Physics, 2016). We also use renormalization, but in order
to treat irrational values of $\beta$, instead of geometric arguments,
we use the renormalization associated to the continued fraction algorithm
and dynamical Ostrowski expansions. This yields a suitable symbolic
coding framework which allows us to reduce the main result to a CLT
for non homogeneous Markov chains.
\end{abstract}

\maketitle

\section{introduction and results}

The main result of this article is a temporal distributional limit
theorem (see Section \ref{subsec:Temporal-and-Spacial} below) for
certain functions over an irrational rotation (Theorem \ref{thm: Main thm}
below). In order to introduce and motivate this result, in the first
section, we first define two types of distributional limit theorems
in the study of dynamical systems, namely spatial and temporal. Temporal
limit theorems in dynamics are the focus of the recent paper \cite{dolgopyat2016temporal}
by D. Dolgopyat and O. Sarig; we refer the interested reader to \cite{dolgopyat2016temporal}
and the references therein for a comprehensive introduction to the
subject, as well as for a list of examples of dynamical systems known
up to date to satisfy temporal distributional limit theorems. In section
\ref{subsec:Beck's-temporal-CLT} we then focus on irrational rotations,
which are one of the most basic examples of low complexity dynamical
systems, and recall previous results on temporal limit theorems for
rotations, in particular Beck's temporal CLT. Our main result in stated
in section \ref{subsec:Main-result-and}, followed by a description
of the structure of the rest of the paper in section \ref{subsec:Proof-tools-and}.

\subsection{Temporal and Spatial Limits in dynamics. \label{subsec:Temporal-and-Spacial}}

Distributional limit theorems appear often in the study of dynamical
systems as follows. Let $X$ be a complete separable metric space,
$m$ a Borel probability measure on $X$ and denote by $\mathcal{B}$
is the Borel $\sigma$-algebra on $X$. Let $T:\,X\rightarrow X$
be a Borel measurable map. We call the quadruple $\left(X,\mathcal{B},m,T\right)$
a \textit{probability preserving dynamical system} and assume that
$T$ is ergodic with respect to $m$. Let $f:\,X\rightarrow\bbR$
be a Borel measurable function and set 
\[
S_{n}\left(T,f,x\right):=\sum_{k=0}^{n-1}f\circ T^{k}\left(x\right)
\]

We will also use the notation $S_{n}\left(x\right)$, or $S_{n}\left(f,x\right)$
instead of $S_{n}\left(T,f,x\right)$, when it is clear from the context,
what is the underlying transformation or function. The function $S_{n}(x)$
is called (the $n^{th}$) \emph{Birkhoff sum} (or also ergodic sum)
of the function $f$ over the transformation $T$. The study of Birkhoff
sums, their growth and their behavior is one of the central themes
in ergodic theory\textcolor{black}{. When the transformation $T$
is }\textcolor{black}{\emph{ergodic}}\textcolor{black}{{} with respect
to $m,$ by the }\textcolor{black}{\emph{Birkhoff ergodic theorem}}\textcolor{black}{,
for any $f\in L^{1}(X,m)$, for $m$-almost every $x\in X$, $S_{n}(f,x)/n$
converges to $\int fdm$ as $n$ grows; equivalently, one can say
that the random variables $X_{n}:=f\circ T^{n}$ where $x$ is chosen
randomly according to the measure $m$, }\textcolor{black}{\emph{satisfy
the strong law of large numbers.}}\textcolor{black}{{} We will now introduce
some limit theorems which allow to study the error term in the Birkhoff
ergodic theorem.}

The function $f$ is said to satisfy a \textit{spatial distributional
limit theorem} (\emph{spatial} DLT) if there exists a random variable
with no atoms $Y$, and sequences of constants $A_{n},B_{n}\in\bbR$,
$B_{n}\rightarrow\infty$, such that the random variables $\frac{S_{n}\left(x\right)-A_{n}}{B_{n}}$,
where $x$ is chosen randomly according to the measure $m$, converge
in distribution to $Y$. In this case we write 
\[
\frac{S_{n}-A_{n}}{B_{n}}\overset{dist}{\longrightarrow}Y.
\]
It is the case that many \emph{hyperbolic }dynamical systems, under
some regularity conditions on $f$, satisfy a spatial DLT with the
limit being a Gaussian random variable. In the cases that we have
in mind, the rate of mixing of the sequence of random variables $X_{n}:=f\circ T^{n}$
is sufficiently fast, in order for them to satisfy the Central Limit
Theorem (CLT). On the other hand, in many classical examples of dynamical
systems with \emph{zero entropy,} for which the random variables $X_{n}:=f\circ T^{n}$
are highly correlated, the spatial DLT fails if $f$ is sufficiently
regular. For example, this is the case when $T$ is an irrational
rotation and $f$ is of bounded variation. 

Perhaps surprisingly, many examples of dynamical systems with zero
entropy satisfy a CLT when instead of averaging over the space $X$,
one considers the Birkhoff sums $S_{n}\left(x_{0}\right)$ over a
\emph{single orbit} of some fixed initial condition $x_{0}\in X$.
Fix an initial point $x_{0}\in X$ and consider its orbit under $T$.
One can define a sequence of \emph{occupation measures} on $\bbR$
by 
\[
\nu_{n}\left(F\right):=\frac{1}{n}\#\left\{ 1\leq k\leq n:\quad S_{k}\left(x_{0}\right)\in F\right\} 
\]
for every Borel measurable $F\subset\bbR$. One can interpret the
quantity $\nu_{n}\left(F\right)$ as the fraction of time that the
Birkhoff sums $S_{k}\left(x_{0}\right)$ spend in the set $F$, up
to time $n$. Let $Y_{n}$ be a sequence of random variables distributed
according to $\nu_{n}$. We say that the pair $(T,f)$ satisfies a
\emph{temporal distributional limit theorem} (\emph{temporal} DLT)
along the orbit of $x_{0}$, if there exists a random variable with
no atoms $Y$, and two sequences $A_{n}\in\bbR$ and $B_{n}\rightarrow\infty$
such that $(Y_{n}-A_{n})/B_{n}$ converges in distribution to $Y$.
In other words, the pair $\left(T,f\right)$ satisfies a temporal
DLT along the orbit of $x_{0}$, if 
\[
\frac{1}{n}\#\left\{ 1\leq k\leq n:\,\frac{S_{k}\left(x_{0}\right)-A_{n}}{B_{n}}<a\right\} \overset{n\rightarrow\infty}{\longrightarrow}Prob\left(Y<a\right)
\]
for every $a\in\bbR$. If the limit $Y$ is a Gaussian random variable,
we call this type of behavior a \emph{temporal CLT} along the orbit
of $x_{0}$. Note, that this type of result may be interpreted as
convergence in distribution of a sequence of normalized random variables,
obtained by considering the Birkhoff sums $S_{k}\left(x_{0}\right)$
for $k=1,..,n$ and choosing $k$ randomly uniformly. 

\subsection{Beck's temporal CLT and its generalizations\label{subsec:Beck's-temporal-CLT}}

One example of occurrence of a temporal CLT in dynamical systems with
zero entropy is the following result by Beck, generalizations of which
are the main topic of this paper. Let us denote by $R_{\alpha}$ the
rotation on the interval $\mathbb{T}=\bbR\setminus\bbZ$ by an irrational
number $\alpha\in\mathbb{R}$, given by 
\[
R_{\alpha}(x)=x+\alpha\mod1.
\]
Let $f_{\beta}:\,\mathbb{T}\to\mathbb{R}$ be the indicator of the
interval $[0,\beta)$ where $0<\beta<1$, rescaled to have mean zero
with respect to the Lebesgue measure on $\mathbb{T}$, namely 
\[
f_{\beta}(x)=\ind_{\left[0,\beta\right)}\left(x\right)-\beta.
\]
The sequence $\left\{ S_{n}\right\} $ of random variables given by
the Birkhoff sums $S_{n}(x)=S_{n}\left(R_{\alpha},f_{\beta},x\right)$
, where $x$ is taken uniformly with respect to the Lebesgue measure,
is sometimes referred to in the literature as the \emph{deterministic
random walk }driven by an irrational rotation (see for example \cite{avila2015visits}).

Beck proved \cite{beck2010randomness,beck2011randomness} that if
$\alpha$ is a quadratic irrational, \textcolor{black}{and} $\beta$
is rational, then the pair $(R_{\alpha},f_{\beta})$ satisfies a temporal
DLT along the orbit of $x_{0}=0$. More precisely, he shows that there
exist constants $C_{1}$ and $C_{2}$ such that for all $a,b\in\bbR$,
$a<b$ \textcolor{black}{
\[
\frac{1}{n}\#\left\{ 1\leq k\leq n:\,\frac{S_{k}\left(R_{\alpha},f_{\beta},0\right)-C_{1}\log n}{C_{2}\sqrt{\log n}}\in\left[a,b\right]\right\} \rightarrow\frac{1}{\sqrt{2\pi}}\int\limits _{a}^{b}e^{-\frac{x^{2}}{2}}dx.
\]
}

\textcolor{black}{Beck's CLT relates to the theory of discrepancy
in number theory as follows. If $\alpha\in\mathbb{R}$ is irrational,
by unique ergodicity of the rotation $R_{\alpha}$, the sequence of
$\left\{ j\alpha\right\} $ is }\textcolor{black}{\emph{equidistributed
modulo one}}\textcolor{black}{, i.e. in particular, for any}\textbf{\textcolor{black}{{}
$\beta\in\left[0,1\right]$}}\textcolor{black}{{} if we set 
\[
N_{k}(\alpha,\beta):=\#\left\{ 0\leq j<k\:|\quad0\leq j\alpha\mod1<\beta\right\} ,
\]
then $N_{k}\left(\alpha,\beta\right)/k$ converges to $\beta$, or,
equivalently, $N_{k}\left(\alpha,\beta\right)=k\beta+o(k)$. Discrepancy
theory concerns the study of the error term in the expression $N_{k}\left(\alpha,\beta\right)=k\beta+o(k)$.
Beck's result hence says that, when $\alpha$ is a }\textcolor{black}{\emph{quadratic
irrational}}\textcolor{black}{{} and $\beta$ is}\textcolor{black}{\emph{
rational, }}\textcolor{black}{the error term $\overline{N_{k}}\left(\alpha,\beta\right):=N_{k}\left(\alpha,\beta\right)-k\beta$
, when $k$ is chosen uniformly in $\left\{ 1,\dots,n\right\} $,
can be normalized so that it converges to the standard Gaussian distribution
as $n$ grows to infinity.}

\textcolor{black}{Let us also remark that the Birkhoff sums in the
statement of Beck's theorem are related to the dynamics of the map
$T_{f_{\beta}}:\mathbb{T}\times\mathbb{R}\rightarrow\mathbb{T\times R}$,
defined by 
\[
T_{f_{\beta}}\left(x,y\right)=\left(R_{\alpha}\left(x\right),y+f_{\beta}\left(x\right)\right),\qquad\text{\ensuremath{\left(x,y\right)\in T\times\mathbb{R},}}
\]
since one can see that the form of the iterates of $T_{f}$ is $T_{f}^{n}\left(x,y\right)=\text{\ensuremath{\left(R_{\alpha}^{n}\left(x\right),y+S_{n}\text{\ensuremath{\left(f_{\beta},x\right)}}\right)}.}$
This skew product map has been studied as one basic example in infinite
ergodic theory and there is a long history of results on it, starting
from ergodicity (see for example \cite{schmidt1978cylinder,conze1976ergodicite,aaronson1982visitors,oren1983ergodicity,avila2015visits,aaronson2016discrepancy}). }

\textcolor{black}{Recently, in \cite{avila2015visits}, a new proof
of Beck's theorem for the special case where $\beta=\frac{1}{2}$,
which uses dynamical and geometrical renormalization tools. It is
crucially based on the interpretation of the corresponding skew-product
map $T_{f_{1/2}}$ as the Poincaré map of a flow on the }\textcolor{black}{\emph{staircase}}\textcolor{black}{{}
periodic surface, which was noticed and pointed out in \cite{hooper2013dynamics}.}
In \cite{dolgopyat2016temporal} this method is generalized to show
that for any initial point $x$, any $\alpha$ quadratic irrational
and any\emph{ rational} $\beta$, there exists a sequence $A_{n}:=A_{n}\left(\alpha,\beta,x\right)$
and a constant $B:=B\left(\alpha,\beta\right)$ such that 
\[
\frac{1}{n}\#\left\{ 1\leq k\leq n:\,\frac{S_{k}\left(R_{\alpha},f_{\beta},x\right)-A_{n}}{B\sqrt{\log n}}\in\left[a,b\right]\right\} \rightarrow\frac{1}{\sqrt{2\pi}}\int\limits _{a}^{b}e^{-\frac{x^{2}}{2}}dx
\]
for all $a,b\in\bbR$, $a<b$.  Dolgopyat and Sarig showed us how to use the staircase method to prove the temporal CLT also in the case when $\alpha$ \emph{is badly
approximable}, for a.e. $x$ and $\beta=1/2$  (private communication), but their methods do not apply  to the more general class of $\beta$s that we treat in this paper. They also informed us that they can show that the temporal CLT does not hold for a.e. $(\alpha,x)$ and $\beta=1/2$.


\subsection{Main result and comments\label{subsec:Main-result-and}}

The main result of this paper is the following generalization of Beck's
temporal CLT, in which we consider certain irrational values of $\beta$
and badly approximable values of $\alpha$. Let us recall that $\alpha$
is\emph{ badly approximable} (or equivalently, $\alpha$ is of\emph{
bounded type}) if there exists a constant $c>0$ such that $\left|\alpha-p/q\right|\geq c/\left|q\right|$
for any $p,q$ , $q\neq0$. Equivalently, $\alpha$ is \emph{badly
approximable} if the continued faction entries of $\alpha$ are uniformly
bounded. For $\alpha\in\left(0,1\right)\setminus\bbQ$ let us say
that $\beta$ is \emph{badly approximable with respect to} $\alpha$
if there exists a constant $C>0$ such that

\begin{equation}
\left|q\alpha-\beta-p\right|>\frac{C}{\left|q\right|}\qquad\mathrm{for\,all\,p\in\mathbb{Z},\,q\in\mathbb{\bbZ}\setminus\left\{ 0\right\} .}\label{eq:BAbeta}
\end{equation}
One can show that given a badly approximable $\alpha$, the set of
$\beta$ which are badly approximable with respect to $\alpha$ have
full Hausdorff dimension. 
\begin{thm}
\label{thm: Main thm}Let $0<\alpha<1$ be a badly approximable irrational
number. For every $\beta$ badly approximable with respect to $\alpha$
and every $x\in\mathbb{T}$ there exists a sequence of centralizing
constants $A_{n}:=A_{n}\left(\alpha,\beta,x\right)$ and a sequence
of normalizing constants $B_{n}:=B_{n}\left(\alpha,\beta\right)$
such that for all $a<b$ 
\[
\frac{1}{n}\#\left\{ 1\leq k\leq n:\,\frac{S_{k}\left(R_{\alpha},f_{\beta},x\right)-A_{n}}{B_{n}}\in\left[a,b\right]\right\} \rightarrow\frac{1}{\sqrt{2\pi}}\int\limits _{a}^{b}e^{-\frac{x^{2}}{2}}dx.
\]
\end{thm}
In other words, for every $\alpha$ badly approximable, any $\beta$
badly approximable with respect to $\alpha$ the pair $\left(R_{\alpha},f\right)$
satisfies the temporal CLT along the orbit of any $x\in\mathbb{T}$.
Note that the centralizing constants depend on $x$, while the normalizing
constants do not. We will see in Section \ref{subsec:Positivity-of-products}
that badly approximable numbers with respect to $\alpha$ can be explicitly
described in terms of their \emph{Ostrowski expansion}, using an adaptation
of the continued fraction algorithm in the context of non homogenous
Diophantine Approximation. Let us recall that quadratic irrationals
are in particular badly approximable. Moreover, when\emph{ $\alpha$}
is badly approximable, it follows from definition that any \emph{rational}
number $\beta$ is badly approximable with respect to $\alpha$. Thus,
this theorem, already in the special case in which $\alpha$ is assumed
to be a quadratic irrational, since it includes \emph{irrational}
values of $\beta$, gives a strict generalization of the results mentioned
above. As we already pointed out, the temporal limit theorem, fails
to hold for almost every value of $\alpha$. It would be interesting to see whether a temporal
CLT holds for a larger class of values of $\beta$. 

While the proof of Theorem \ref{thm: Main thm} was inspired and motivated
by an insight of Dolgopyat and Sarig and based, as theirs, on renormalization,
we stress that our renormalization scheme and the formalism that we
develop is different. As remarked in the previous section, the proof
of Beck's theorem in \cite{avila2015visits,dolgopyat2016temporal}
exploits a geometric renormalization which is based on the link with
the staircase flow and the existence of affine diffeomorphisms which
renormalize certain directions of directional flows on this surface.
This geometrical insight, unfortunately, as well as the interpretation
of the map $T_{f}$ as the Poincaré map of a staricase flow, breaks
down when $\beta$ is not rational. Our proof does not rely on this
geometric picture, but uses only the more classical renormalization
given by the continued fraction algorithm for rotations, with the
additional information encoded by Ostrowski expansions in the context
of non homogeneous Diophantine approximations (see Section \ref{sec:RENORMALIZATION}).
This renormalization allows to encode the dynamics symbolically and
reduce it to the formalism of adic and Vershik maps \cite{vershik1997adic}). 

There is a large literature of results on limiting distributions for
entropy zero dynamical systems, see for example \cite{BufetovLimitThmTranslationflows14,BufetovSolomyak13,BufetovForni11,DolgoFayadLimit_Toral_Trans_15,MR3277200,SinaiUlcigrai08}.
Let us mention two recent results in the context of substitution systems
which are related to our work. Bressaud, Bufetov and Hubert proved
in \cite{bressaud2014deviation} a spatial CLT for substitutions with
eigenvalues of modulus one along a subsequence of times. In the same
context (substitutions with eigenvalues of modulus one), Paquette
and Son \cite{paquette2015birkhoff} recently also proved a \emph{temporal
}CLT. In \cite{aaronson2017rational} a temporal CLT over quadratic
irrational rotations and $\bbR^{d}$ valued, piecewise constant functions
with rational discontinuities, is shown to hold along subsequences.

While we wrote this paper specifically for deterministic random walks
driven by rotations, there are other entropy zero dynamical systems
where this formalism applies and for which one can prove temporal
limit theorems using similar techniques. For example, in work in progress,
we can prove temporal limit theorems also for certain linear flows
on infinite translation surfaces and some cocycles over interval exchange
transformation and more in general for certain $\mathcal{S}-$\emph{adic}
systems (which are non-stationary generalizations of substitution
systems, see \cite{berthe2013beyond}). 

\subsection{Proof tools and sketch and outline of the paper\label{subsec:Proof-tools-and}}

In Section \ref{sec:RENORMALIZATION} we introduce the renormalization
algorithm that we use, as a key tool in the proofs: this is essentially
the classical multiplicative continued fraction algorithm, with additional
data which records the relative position of the break point $\beta$
of the function $f_{\beta}$ under renormalization. This renormalization
acts on the underlying parameter space to be defined in what follows,
as a (skew-product) extension of the Gauss map, and it produces simultaneously
the continued fraction expansion entries of $\alpha$ and the Ostrowski
expansion entries of $\beta$. Variations on this skew product have
been studied by several authors (see in particular \cite{arnoux2001scenery,shunji1986some})
and it is well known that it is related to a section of the diagonal
flow on the space of affine lattices (as explained in detail in \cite{arnoux2001scenery}).
In sections \ref{subsec: Symbolic Coding} and \ref{subsec: Markov chain towers}
we explain how the renormalization algorithm provides a way of encoding
dynamics symbolically in terms of a Markov chain. More precisely,
the dynamics of the map $R_{\alpha}$ we are interested in translates
in symbolic language to the adic or Vershik dynamics (on a Bratelli
diagram given by the Markov chain), as explained in section. The original
function $f_{\beta}$ defines under renormalization a sequence of
induced functions (which correspond to Birkhoff sums of the function
$f_{\beta}$ at first return times, called special Birkhoff sums in
the terminology introduced by \cite{marmi2005cohomological}). The
Birkhoff sums of the function $f_{\beta}$ can be then decomposed
into sums of special Birkhoff sums. This formalism and the symbolic
coding allows to translate the study of the temporal visit distribution
random variable to the study of a non-homogeneous Markov chain, see
section \ref{subsec: Markov chain-1}. In Section \ref{sec:The-CLT-for-Markov chains}
we provide sufficient conditions for a non-homogeneous Markov chain
to satisfy the CLT. Finally, in Section \ref{sec:Proof-of-the TCLT}
we prove that these conditions are satisfied for the Markov chain
modeling the temporal distribution random variables.

\section{renormalization\label{sec:RENORMALIZATION}}

\subsection{Preliminaries on continued fraction expansions and circle rotations}

Let $\mathcal{G}$:$(0,1)\to(0,1)$ be the Gauss map, given by $\mathcal{G}(x)=\left\{ \nicefrac{1}{x}\right\} $,
where $\{\cdot\}$ denotes the fractional part. Recall that a regular
continued fraction expansion of $\alpha\in\left(0,1\right)\setminus\bbQ$
is given by 
\[
\alpha=\cfrac{1}{a_{0}+\cfrac{1}{a_{1}+\dots}}
\]
where $a_{i}:=a\left(\alpha_{i}\right)=\left[\frac{1}{\alpha_{i}}\right]$
and $\alpha_{i}:=G^{i}(\alpha)$=$\left\{ \frac{1}{\alpha_{i-1}}\right\} $.
In this case we write $\alpha=\left[a_{0},a_{1},...\right]$ . Setting
$q_{-1}=1$, $q_{0}=a_{0}$, $q_{n}=a_{n}q_{n-1}+q_{n-2}$ for $n\geq1$,
and $p_{-1}=0$, $p_{0}=1$, $p_{n}=a_{n}p_{n-1}+p_{n-2}$ for $n\geq1$
we have $\gcd\left(p_{n},q_{n}\right)=1$ and 
\[
\cfrac{1}{a_{0}+\cfrac{1}{a_{1}+\dots\cfrac{1}{a_{n}}}}=\frac{p_{n}}{q_{n}}.
\]

Let $\alpha\in\left(0,1\right)\setminus\bbQ$, $\mathbb{T}:=\bbR/\bbZ$
and $R_{\alpha}:\mathbb{T}\rightarrow\mathbb{T}$ be the irrational
rotation given by$R_{\alpha}:=x+\alpha\mod1$. Then the Denjoy-Koksma
inequality \cite{herman1979conjugaison,katznelson1977sigma} states
that if $f:\mathbb{T}\rightarrow\bbR$ is a function of bounded variation,
then for any $n\in\bbN$, 
\begin{equation}
\sup\left\{ \left|f\circ R_{\alpha}^{q_{n}}\left(x\right)\right|:\ x\in\mathbb{T}\right\} \leq\bigvee\nolimits _{\mathbb{T}}f\label{eq: DK inequality}
\end{equation}
where $\bigvee_{\mathbb{T}}f$ is the variation of $f$ on $\mathbb{T}$. 

In this section we define the dynamical renormalization algorithm
we use in this paper, which is an extension of the classical continued
fraction algorithm and hence of the Gauss map. This algorithm gives
a dynamical interpretation of the notion of \emph{Ostrowski expansion}
of $\beta$ relative to $\alpha$ in non-homogeneous Diophantine approximation.
We mostly follow the conventions of the paper \cite{arnoux2001scenery}
by Arnoux and Fisher, in which the connection between this renormalization
and homogeneous dynamics (in particular the geodesic flow on the space
of lattices with a marked point, which is also known as \emph{the
scenery flow)} is highlighted. As in \cite{arnoux2001scenery} we
use a different convention for rotations on the circle. Let $\alpha\in\left(0,1\right)\setminus\bbQ$,
$I=\left[-1,\alpha\right)$ and let $T_{\alpha}\left(x\right):\left[-1,\alpha\right)\rightarrow\left[-1,\alpha\right)$
be defined by 
\begin{equation}
T_{\alpha}\left(x\right)=\begin{cases}
x+\alpha & x\in\left[-1,0\right)\\
x-1 & x\in\left[0,\alpha\right)
\end{cases}\label{eq: Irr. Rotaion}
\end{equation}
Note that $T_{\alpha}$ may also be viewed as a rotation on the circle
$\bbR/\sim$ where the equivalence relation $\sim$ on $\bbR$ is
given by $x\sim y\iff x-y\in\left(1+\alpha\right)\bbZ$. It is conjugate
to the standard rotation $R_{\alpha'}$ on $\mathbb{T}$ , where $\alpha'=\frac{\alpha}{1+\alpha}$,
by the map $\psi(x)=\left(\alpha'+1\right)x-1$ which maps the unit
interval $[0,1]$ to the interval $[-1,\alpha]$.
\begin{rem}
\label{rem: Circle maps}In what follows, we slightly abuse notation
by not distinguishing between the transformation $T_{\alpha}$ and
the transformation defined similarly on the interval $\left(-1,\alpha\right]$
by 
\[
T'_{\alpha}\left(x\right)=\begin{cases}
x+\alpha & x\in\left(-1,0\right];\\
x-1 & x\in\left[0,\alpha\right);
\end{cases}
\]
when viewed as transformations on the circle, $T_{\alpha}$ and $T_{\alpha'}$
coincide. 
\end{rem}
Note that given an irrational rotation $R_{\alpha}$, we can assume
without loss of generality that $\alpha<\frac{1}{2}$ (otherwise consider
the inverse rotation by $1-\alpha$). If we set 
\begin{equation}
\alpha_{0}:=\frac{\alpha}{1-\alpha}\label{eq:Correcpondence between alpha and alpha_zero}
\end{equation}
 then $\mathcal{G}(\alpha)=\mathcal{G}^{i}(\alpha_{0})$ for any $i\in\mathbb{N}$
and thus, apart from the first entry, the continued fraction entries
of $\alpha$and $\alpha_{0}$ coincide. If $a_{0}$, $a_{0}'$ are
correspondingly the first entries in the expansion of $\alpha_{0}$
and $\alpha$, then $a_{0}=a_{0}'+1$. Furthermore, given $\beta\in\left(0,1\right)$,
let 
\begin{equation}
\beta_{0}:=\left(\alpha+1\right)\beta-1.\label{eq: Correcspondence between beta and beta_zero}
\end{equation}
Then the mean zero with a discontinuity at $\beta_{0}$, given by

\begin{equation}
\varphi\left(x\right)=\ind_{\left[-1,\beta_{0}\right)}\left(x\right)-\frac{\beta_{0}+1}{\alpha_{0}+1}\label{eq: Cocycle function}
\end{equation}
is the function that corresponds to the function $f_{\beta}$ in the
introduction under the conjugation between $R_{\alpha}$ and $T_{\alpha_{0}}$.
Therefore, we are interested in the Birkhoff sums 
\begin{equation}
\varphi_{n}\left(x\right)=\sum_{k=0}^{n-1}\varphi\left(T_{\alpha_{0}}^{k}\left(x\right)\right).\label{eq:defBS}
\end{equation}
Henceforth, unless explicitly stated otherwise, we work with the transformation
$T_{\alpha_{0}}$. The sequences $\left(a_{n}\right)_{n=0}^{\infty}$,
$\left(\frac{p_{n}}{q_{n}}\right)_{n=0}^{\infty}$ will correspond
to the sequence of entries and the sequence of partial convergents
in the continued fraction expansion of $\alpha=\frac{\alpha_{0}}{1+\alpha_{0}}$. 

We denote by $\lambda$ the Lebesgue measure on $\left[-1,\alpha_{0}\right)$
normalized to have total mass $1$. 

\subsection{\label{subsec:Continued-fraction-renormalization}Continued fraction
renormalization and Ostrowski expansion }

The renormalization procedure is an inductive procedure, where at
each stage we induce the original transformation $T_{\alpha_{0}}$
onto a subinterval of the interval we induced upon at the previous
stage. We denote by $I^{\left(n\right)}$ the nested sequence of intervals
which we induce upon, and by $T^{\left(n\right)}$ the first return
map of $T_{\alpha_{0}}$ onto $I^{\left(n\right)}$. The nested sequence
of intervals $I^{\left(n\right)}$ is chosen in such a way that the
induced transformations $T^{\left(n\right)}$ are all irrational rotations.
The next paragraph describes a step of induction given an irrational
rotation $T_{\alpha_{n}}$ on the interval $I_{n}=\left[-1,\alpha_{n}\right)$
defined by (\ref{eq: Irr. Rotaion}). The procedure is then iterated
recursively by rescaling and performing the induction step once again.
In general, we keep to the convention that we use $n$ as a superscript
to denote objects related to the non-rescaled $n$th step of renormalization,
and as a subscript for the rescaled version. 

\subsubsection*{One step of renormalization}

For an irrational $\alpha_{n}\in\left(0,1\right)$ let $I_{n}:=[-1,\alpha_{n})$,
$T_{\alpha_{n}}:I_{n}\rightarrow I_{n}$ defined by the formula in
(\ref{eq: Irr. Rotaion}) and $\beta_{n}\in I_{n}$. Then $T_{\alpha_{n}}$
is an exchange of two intervals of lengths $\alpha_{0}$ and $1$
respectively (namely $\left[0,\alpha_{0}\right)$ and $\left[-1,0\right)$).
The renormalization step consists of inducing $T_{\alpha_{n}}$ onto
an interval \emph{$I_{n}'$, }where\emph{ $I_{n}'$ }is obtained by\emph{
cutting} a half-open interval of size $\alpha_{0}$ from the left
endpoint of the interval $I_{0}$, i.e. $-1$, as many times as possible
in order to obtain an interval of the from $[-\alpha'_{n},\alpha_{n})$
containing zero. More precisely, let 
\[
a_{n}=[1/\alpha_{n}],\qquad\alpha'_{n}=1-a_{n}\alpha_{n}
\]
so that $[-1,0)$ contains exactly $a_{n}$ intervals of lengths $\alpha_{n}$
plus an additional remainder of length $0<\alpha'_{n}<\alpha_{n}$
(see Figure \ref{fig:One-step-of Ostrowski}). If $\beta_{n}\in[-1,-\alpha'_{n})$,
let $1\leq b_{n}\leq a_{n}$ be such that $\beta_{n}$ belongs to
the $b_{n}^{th}$ copy of the interval which is cut, otherwise set
$b_{n}:=0$, i.e. define 
\begin{equation}
b_{n}:=\begin{cases}
[(\beta_{n}-(-1))/\alpha_{n}]+1=[(1+\beta_{n})/\alpha_{n}]+1 & \mathrm{if}\,\beta_{n}\in[-1+(b_{n}-1)\alpha_{n},-1+b_{n}\alpha{}_{n})\\
0 & \mathrm{if}\,\beta_{n}\in[-\alpha'_{n},\alpha_{n}).
\end{cases}\label{eq: Ostrowsky b_n}
\end{equation}

\begin{figure}[!tbh]
\includegraphics[width=0.6\textwidth]{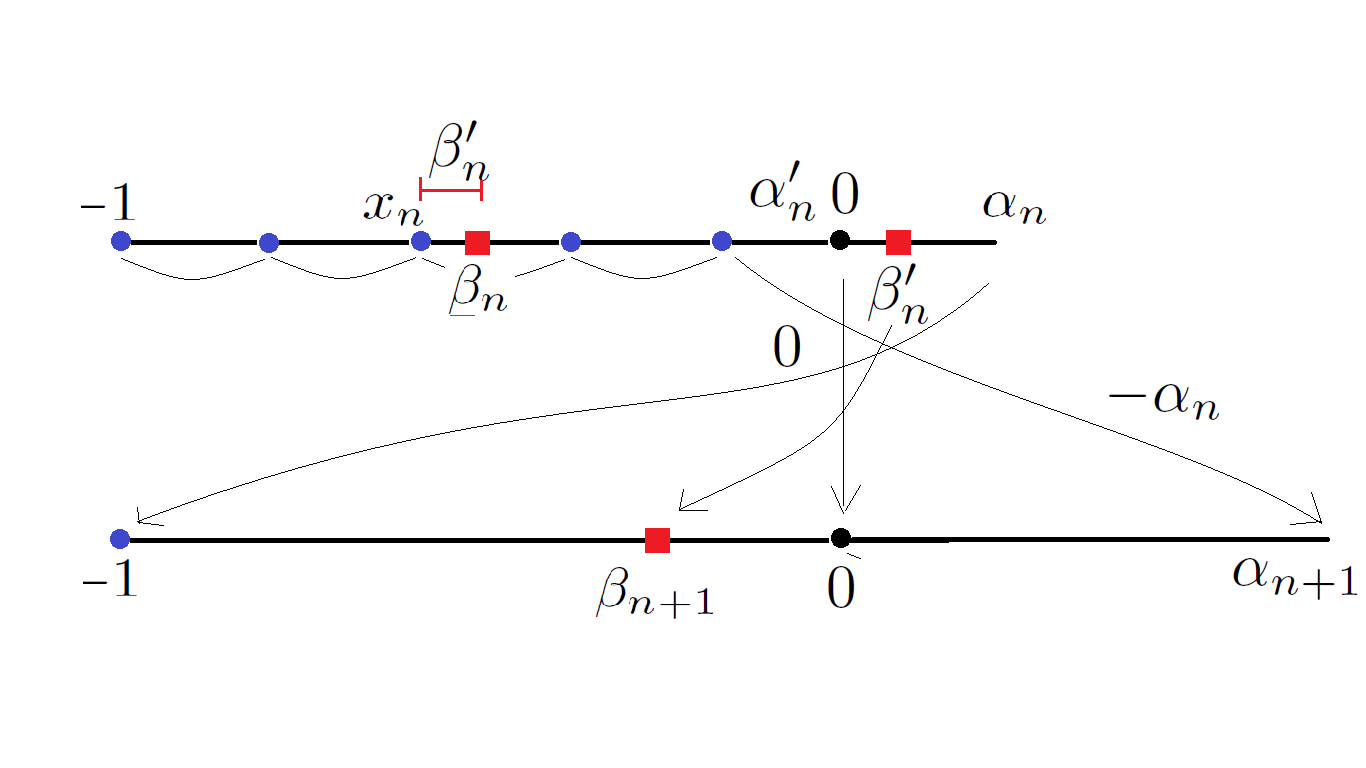}

\caption{\label{fig:One-step-of Ostrowski}One step of the Ostrowski renormalization
algorithm.}
\end{figure}

For $b_{n}\geq1$, let us define $x_{n}$ to be the left endpoint
of the copy of the interval which contains $\beta_{n}$, otherwise,
if $b_{n}=0$, set $x_{n}:=0$; let also $\beta_{n}':=\beta_{n}-x_{n}$,
so that if $b_{n}\geq1$ then $\beta_{n}'$ is the distance of $\beta_{n}$
from the left endpoint of the interval which contains it (Figure \ref{fig:One-step-of Ostrowski}).
In formulas

\begin{equation}
x_{n}:=\begin{cases}
-1+(b_{n}-1)\alpha_{n} & if\ b_{n}\geq1\\
0 & if\ b_{n}=0
\end{cases},\qquad\beta'_{n}:=\begin{cases}
\beta_{n}+1-(b_{n}-1)\alpha_{n} & \mathrm{if}\ b_{n}\geq1\\
\beta_{n} & \mathrm{if}\ b_{n}=0
\end{cases}.\label{eq:xn}
\end{equation}
Notice that $x_{n}=T_{\alpha_{n}}^{b_{n}}(0)$ and hence in particular
it belongs to the segment $\left\{ 0,T_{\alpha_{n}}(0),\dots,T_{\alpha_{n}}^{a_{n}}(0)\right\} $
of the orbit of $0$ under $T_{\alpha_{n}}$. 

Let $I_{n}'=[-\alpha'_{n},\alpha_{n})$ and note that $\beta_{n}'\in I_{n}'$
and that the induced transformation obtained as the first return map
of $T_{\alpha_{n}}$ on $I_{n}'$ is again an exchange of two intervals,
a short one $\left[-\alpha_{n}',0\right)$ and a long one $\left[0,\alpha_{n}\right)$.
Hence, if we \emph{renormalize} and \emph{flip }the picture by multiplying
by $-\alpha_{n}$, the interval $I_{n}'$ is mapped to $I_{n+1}:=\left(-1,\alpha_{n+1}\right]$
, where 

\[
\alpha_{n+1}:=\frac{\alpha'_{n}}{\alpha_{n}}=\frac{1-a_{n}\alpha_{n}}{\alpha_{n}}=\frac{1}{\alpha_{n}}-\left[\frac{1}{\alpha_{n}}\right]=G(\alpha_{n})
\]

and the transformation $T_{\alpha_{n}}$ as first return on the interval
$I_{n}'$ is conjugated to $T_{\alpha_{n+1}}$. We then set 
\[
\beta_{n+1}:=-\frac{\beta_{n}'}{\alpha_{n}}=-\frac{\beta_{n}-x_{n}}{\alpha_{n}},
\]
so that $\beta_{n+1}\in I_{n+1}.$ Thus we have defined $\alpha_{n+1}$,
$\beta_{n+1}$ and $T_{\alpha_{n+1}}$ and completed the description
of the step of induction. 

Notice that by definition of $\beta_{n}'$ and $\beta_{n+1}$we have
that 

\begin{equation}
\beta_{n}=\begin{cases}
-1+\left(b_{n}-1\right)\alpha_{n}-\alpha_{n}\beta_{n+1}=x_{n}-\alpha_{n}\beta_{n+1} & b_{n}\geq1\\
-\alpha_{n}\beta_{n+1} & b_{n}=0
\end{cases}.\label{eq:Recursive formula for Ostrowski}
\end{equation}

and hence, by using equation (\ref{eq: Ostrowsky b_n}), we get 
\begin{equation}
\beta_{n+1}=\mathcal{H}\left(\alpha_{n},\beta_{n}\right):=\begin{cases}
-\left\{ \frac{\beta_{n}+1}{\alpha_{n}}\right\}  & if\ b_{n}\geq1\\
-\frac{\beta_{n}}{\alpha_{n}} & if\ b_{n}=0.
\end{cases}\label{eq: Ostrowski beta_n}
\end{equation}

Repeating the described procedure inductively, one can prove by induction
the assertions summarized in the next proposition. 
\begin{prop}
\label{prop: Ostrowsky expansion}Let $\alpha^{\left(n\right)}:=\alpha_{0}\cdot...\cdot\alpha_{n}$
where $\alpha_{i}$ are defined inductively from $\alpha_{0}$ by
$\alpha_{n}=\mathcal{G}\left(\alpha_{n-1}\right)$ and set $\alpha^{\left(-1\right)}=1$.
Define a sequence of nested intervals $I^{\left(n\right)}$, $n=0,1,...$,
by $I^{\left(0\right)}:=\left[-1,\alpha^{\left(0\right)}\right)$,
and 
\[
I^{\left(n\right)}:=\begin{cases}
(-\alpha^{\left(n-1\right)},\alpha^{\left(n\right)}] & if\ n\ is\ odd;\\
\left[-\alpha^{\left(n\right)},\alpha^{\left(n-1\right)}\right) & if\ n\ is\ even.
\end{cases}
\]

The induced map $T^{\left(n\right)}$ of $T_{\alpha_{0}}$ on $I^{\left(n\right)}$
is conjugated to $T_{\alpha_{n}}$ on the interval $I_{n}=\left[-1,\alpha_{n}\right)$
if $n$ is even or to $T_{\alpha_{n}}$on $I_{n}=\left(-1,\alpha_{n}\right]$
if $n$ is odd, where the conjugacy is given by $\psi_{n}:I_{n}\rightarrow I^{\left(n\right)}$,
$\psi_{n}\left(x\right)=\left(-1\right)^{n}\alpha^{\left(n-1\right)}\left(x\right)$.

\medskip{}
Let $\beta_{0}\in I^{\left(0\right)}$ and let $(b_{n})_{n}$ and
$(\beta_{n})_{n}$ be the sequences inductively\footnote{Note that given $\beta_{n},$ formulas (\ref{eq: Ostrowsky b_n})
and (\ref{eq: Ostrowski beta_n}) determine first $b_{n}$ and then,
as function of $\beta_{n}$ and $b_{n}$, also $\beta_{n+1}$ and
hence $b_{n+1}$.} by the formulas (\ref{eq: Ostrowsky b_n}) and (\ref{eq: Ostrowski beta_n}).
Then we have
\begin{equation}
\beta_{0}=\sum_{n=0}^{\infty}x^{\left(n\right)},\qquad\text{where }x^{\left(n\right)}=\psi_{n}\left(x_{n}\right)=\begin{cases}
\left(-1\right)^{n}\alpha^{\left(n-1\right)}\left(-1+\left(b_{n}-1\right)\alpha_{n}\right) & 1\leq b_{n}\leq a_{n}\\
0 & b_{n}=0
\end{cases}\label{eq:Ostrowsky}
\end{equation}
and the reminders are given by
\begin{equation}
\left|\beta_{0}-\sum_{k=0}^{n}x^{\left(k\right)}\right|=\left|\beta^{(n+1)}\right|,\qquad\text{where }\beta^{\left(n\right)}:=\psi_{n}\left(\beta_{n}\right)=\text{\ensuremath{\left(-1\right)}}^{n}\alpha^{(n-1)}\beta_{n}.\label{eq: Remainder in Ostrowski expansion}
\end{equation}
\end{prop}
The expansion in (\ref{eq:Ostrowsky}) is an \emph{Ostrowski type}
expansion for $\beta_{0}$ in terms of $\alpha_{0}$. We call the
integers $b_{n}$ the \emph{entries} in the Ostrowski expansion of
$\beta_{0}$. 
\begin{rem}
\label{rem: Ostrowski Remark}Partial approximations in the Ostrowski
expansions have the following dynamical interpretation. It well known
that, for any $n\in\mathbb{\ensuremath{N}}$, the finite segment$\left\{ T_{\alpha_{0}}^{i}\left(0\right):\ i=0,...,q_{n}+q_{n-1}-1\right\} $
of the orbit of $0$ under $T_{\alpha_{0}}$ (which can be thought
of as a rotation on a circle) induce a partition of $[-1,\alpha_{0})$
into intervals of two lengths (see for example \cite{sinai1994topics};
these partitions correspond to the classical Rokhlin-Kakutani representation
of a rotation as two towers over an induced rotation given by the
Gauss map, see also Section \ref{subsec: Tower Structure} and Remark
\ref{rem: Towers generate the sigma-field}). The finite Ostrowski
approximation $\sum_{k=0}^{n}x^{\left(k\right)}$ gives one of the
endpoints of the unique interval of this partition which contains
$\beta_{0}$ (if it is the left or the right one depends on the parity
as well as on whether $b_{n}$ is zero or not). In particular, we
have that

\[
\sum_{k=0}^{n}x^{\left(k\right)}\in\left\{ T_{\alpha_{0}}^{i}\left(0\right):\ i=0,...,q_{n}+q_{n-1}-1\right\} \cup\left\{ \alpha_{0}\right\} ,\qquad n\in\bbN\cup\left\{ 0\right\} .
\]
\end{rem}
\begin{rem}
\label{rem: Infiniteness of Ostrowsky expansion}Since the points
$\alpha^{\left(n\right)}$ are all in the orbit of the point $0$
by the rotation $T_{\alpha_{0}}$, it follows from the correspondence
between $T_{\alpha_{0}}$ and $R_{\alpha}$ that the Ostrowski expansion
of $\beta_{0}$ appearing in the previous proposition is finite, i.e.
$\beta_{0}=\sum_{n=0}^{N}x^{\left(n\right)}$ for some $N\in\bbN$
if and only if $\beta\in\left\{ n\alpha\ \mod1:\ n\in\bbZ\right\} $.
This condition is well known to be equivalent to the function $f_{\beta}$
(and hence also $\varphi$) being a coboundary (see \cite{petersen1973series})
. 
\end{rem}
It follows from the description of the renormalization algorithm that
\[
(\alpha_{n+1},\beta_{n+1})=\hat{\mathcal{G}}(\alpha_{n},\beta_{n}):=\left(\mathcal{G}\left(\alpha_{n}\right),\mathcal{H}\left(\alpha_{n,}\beta_{n}\right)\right)
\]
 where the function $\mathcal{H}$ is defined by (\ref{eq: Ostrowski beta_n}).
The ergodic properties of a variation on the map
\begin{equation}
\hat{\mathcal{G}}:X\rightarrow X,\quad X=\left\{ \left(\alpha,\beta\right):\ \alpha\in\left[0,1\right)\setminus\bbQ,\ \beta\in\left[-1,\alpha\right)\right\} \label{eq:Ghat def}
\end{equation}
were studied among others in \cite{shunji1986some}.

Introduce the functions $a,b:X\to\mathbb{N}$ defined by 

\[
a(\alpha,\beta):=[1/\alpha],\qquad b(\alpha,\beta):=\begin{cases}
[(1+\beta)/\alpha]+1 & \beta\in[-1,-1+a(\alpha,\beta)\alpha)\\
0 & \beta\in[-1+a(\alpha,\beta)\alpha,\alpha)
\end{cases}.
\]

The functions are defined so that the sequences $\left(a_{n}\right)_{n}$
and $\left(b_{n}\right)_{n}$ of continued fractions and Ostrowski
entries are respectively given by $a_{n}=a\left(\hat{\mathcal{G}}^{n}\left(\alpha_{0},\beta_{0}\right)\right)$,
$b_{n}=b\left(\hat{\mathcal{G}}^{n}\left(\alpha_{0},\beta_{0}\right)\right)$
for any $n\in\mathbb{N}.$

By Remark \ref{rem: Infiniteness of Ostrowsky expansion}, the restriction
of the space $X$ to 
\begin{equation}
\tilde{X}:=\left\{ \left(\alpha_{0},\beta_{0}\right)\in X:\ \beta\notin\left\{ n\alpha\ \mod1\right\} \ \text{for }\text{\ensuremath{\alpha}}=\frac{\alpha_{0}}{\alpha_{0}+1},\ \beta=\frac{\beta_{0}+1}{\alpha+1}\right\} \label{eq:Xtilde def}
\end{equation}
 is invariant with respect to $\hat{G}$ and we partition this space
into three sets $X_{G},$ $X_{B_{-}}$, $X_{B_{+}}\subset\tilde{X}$
defined by 
\begin{equation}
\begin{aligned} & X_{G}:=\left\{ \left(\alpha,\beta\right):\ b\left(\alpha,\beta\right)\geq1\right\}  &  & X_{B}:=\left\{ \left(\alpha,\beta\right):\ b\left(\alpha,\beta\right)=0\right\} \\
 & X_{B_{+}}:=X_{B}\cap\left\{ \left(\alpha,\beta\right):\ \beta\geq0\right\}  &  & X_{B_{-}}:=X_{B}\cap\left\{ \left(\alpha,\beta\right):\ \beta<0\right\} 
\end{aligned}
\label{eq:GBB def}
\end{equation}
Explicitly, in terms of the relative position of $\alpha,\beta$,
these sets are given by

\begin{align*}
 & X_{G}=\left\{ \left(\alpha,\beta\right)\in\tilde{X}:\ \beta\in\left[-1,-1+a\left(\alpha,\beta\right)\,\alpha\right)\right\} ,\\
 & X_{B_{-}}=\left\{ \left(\alpha,\beta\right)\in\tilde{X}:\ \beta\in\left[-1+a\left(\alpha,\beta\right)\,\alpha,0\right)\right\} ,\\
 & X_{B_{+}}=\left\{ \left(\alpha,\beta\right)\in\tilde{X}:\ \beta\in\left[0,\alpha\right)\right\} .
\end{align*}
The reason for the choice of names $G$, $B_{-}$, $B_{+}$ for the
thee parts of parameter space, which stand for \emph{Good} ($G$)
and \emph{Bad} ($B$), where \emph{Bad} has two subcases, $B_{-}$
and $B_{+}$ (according to whether $\beta$ is positive or negative),
will be made clear in Section \ref{subsec:Positivity-of-products}. 

\subsection{\label{subsec: Tower Structure}Description of the Kakutani-Rokhlin
towers obtained from renormalization.}

We assume throughout the present Section and Sections \ref{subsec: Symbolic Coding},
\ref{subsec: Markov chain towers} that we are given a fixed pair
$\left(\alpha_{0},\beta_{0}\right)\in\tilde{X}$. The symbols $q_{n}$
used in this Section refer to the denominators of the $n^{th}$ convergent
in the continued fraction expansion of $\alpha$, where $\alpha$
is related to $\alpha_{0}$ via (\ref{eq:Correcpondence between alpha and alpha_zero}).

\smallskip{}
The renormalization algorithm described above defines a nested sequence
of intervals $I^{\left(n\right)}$. We describe here below how the
original transformation $T_{\alpha_{0}}$ can be represented as a
union of\emph{ towers} in a \emph{Kakutani skyscraper} (the definition
is given below) with base $I^{\left(n\right)}$ ; the tower structure
of the skyscraper corresponding to the $(n+1)^{th}$ stage of renormalization
is obtained from the towers of the previous skyscraper corresponding
to the n stage $n^{th}$ by a \emph{cutting and stacking} procedure.
We will use these towers to describe what we call an\emph{ adic} symbolic
coding of the interval $I=\left[-1,\alpha_{0}\right)$ (see section
\ref{subsec: Symbolic Coding}). In what follows, we give a detailed
description of the tower structure and the coding. 

\smallskip{}
Let us first recall that if a measurable set $B\subset[-1,\alpha_{0})$
and a natural integer$h$ are such that the union $\bigcup_{i=0}^{h-1}T_{\alpha_{0}}^{i}B$
is disjoint, we say that the union is a \emph{(dynamical) tower} of
base $B$ and height $h$. The union can indeed be represented as
a tower with $h$ floors, namely $T_{\alpha_{0}}^{i}B$ for $i=0,\dots,h-1,$
so that $T_{\alpha_{0}}$ acts by mapping each point in each level
except the last one, to the point directly above it. A disjoint union
of towers is called a \emph{skyscraper} (see for example \cite{nadkarni1995basic}).
A \emph{subtower} of a tower of base $B$ and height $h$ is a tower
with the same height whose base is a subset$B'\subset B$.

\smallskip{}
As it was explained in the previous section, the induced map of $T_{\alpha_{0}}$
on $I^{(n)}$ is an exchange of two intervals, a\emph{ long} and a\emph{
short} one. If $n$ is even, the long one is given by $\left[-\alpha^{\left(n-1\right)},0\right)$
and the short one by $\left[0,\alpha^{\left(n\right)}\right)$. If
$n$ is odd the long and short interval are respectively given by
$\left[0,\alpha^{\left(n-1\right)}\right)$ and $\left[-\alpha^{\left(n\right)},0\right)$.
In both cases, these are the preimages of the intervals $\left[-1,0\right)$
and $\left[0,\alpha_{n}\right)$ under the conjugacy map $\psi_{n}:I^{\left(n\right)}\rightarrow I_{n}$
given in Proposition \ref{prop: Ostrowsky expansion}. Notice also
that $\beta^{\left(n\right)}=\psi_{n}^{-1}\left(\beta_{n}\right)$,
the non rescaled marked point corresponding to the point $\beta_{n}\in I_{n}$,
further divides the two mentioned subintervals of $I^{\left(n\right)}$
into three, by cutting either the long or the short into two subintervals.
We denote these three intervals $I_{M}^{(n)},I_{L}^{(n)}$ and $I_{S}^{(n)}$,
where the letters $M,L,S$, respectively correspond to \emph{middle}
(M), \emph{long }(L) and \emph{short} (S), and $I_{M}^{(n)}$ denotes
the middle interval, while $I_{L}^{(n)}$ and $I_{S}^{(n)}$ denote
(what is left of) the long one and the short one, after removing the
middle interval. Explicitly, it is convenient to describe the intervals
in terms of the partition $X_{G}$, $X_{B_{-}}$, $X_{B_{+}}$ defined
in the end of the previous section. Thus, set

\[
\begin{aligned} & I_{L}^{\left(n\right)}=\psi_{n}^{-1}\left(\left[-1,\beta_{n}\right)\right), &  & I_{M}^{\left(n\right)}=\psi_{n}^{-1}\left(\left[\beta_{n},0\right)\right), &  & I_{S}^{\left(n\right)}=\psi_{n}^{-1}\left(\left[0,\alpha_{n}\right)\right) & \ensuremath{} & \mathrm{if}\ \left(\alpha_{n},\beta_{n}\right)\in X_{G}\bigcup X_{B_{-}},\\
 & I_{L}^{\left(n\right)}=\psi_{n}^{-1}\left(\left[-1,0\right)\right), &  & I_{M}^{\left(n\right)}=\psi_{n}^{-1}\left(\left[0,\beta_{n}\right)\right), &  & I_{S}^{\left(n\right)}=\psi_{n}^{-1}\left[\beta_{n},\alpha_{n}\right) &  & \mathrm{if}\ \left(\alpha_{n},\beta_{n}\right)\in X_{B_{+}}.
\end{aligned}
\]

We claim that the first return time of $T_{\alpha_{0}}$ to the interval
$I^{\left(n\right)}$ is constant on the subintervals $I_{L}^{\left(n\right)}$,
$I_{M}^{\left(n\right)}$ and $I_{S}^{\left(n\right)}$. Moreover,
the first return time over $I_{L}^{\left(n\right)}$ and $I_{S}^{\left(n\right)}$
equals to $q_{n}$ and $q_{n-1}$ respectively, while the first return
time over $I_{M}^{\left(n\right)}$ equals either $q_{n}$ or $q_{n-1}$,
depending on whether $\beta^{\left(n\right)}\in\left[-\alpha^{\left(n-1\right)},0\right)$
or $\beta^{\left(n\right)}\in\left[0,\alpha^{\left(n\right)}\right)$
and hence on whether the middle interval was cut from the long or
the short interval respectively. For $J\in\left\{ L,M,S\right\} $,
let us denote by $h_{J}^{\left(n\right)}$ the first return time of
$I_{J}^{\left(n\right)}$ to $I^{\left(n\right)}$ under $T_{\alpha_{0}}$
and let us denote by $Z_{J}^{\left(n\right)}$ the tower with base
$I_{J}^{\left(n\right)}$ and height $h_{J}^{\left(n\right)}$. 

\smallskip{}
Let us how describe how the tower structure at stage $n+1$ of the
renormalization is related to the tower structure at stage $n$. We
will describe in detail as an example the particular case where $n$
is odd and $\beta^{\left(n\right)}\in\left[-\alpha^{\left(n-1\right)},-\alpha^{\left(n\right)}\right)$
(i.e. $\beta^{\left(n\right)}\notin I^{\left(n+1\right)}$), or equivalently
$\left(\alpha_{n},\beta_{n}\right)\in X_{G}$ (see also Figure \ref{fig:The-tower-structure}).
The other cases are summarized in Proposition \ref{prop: Tower Structure}
below. In the considered case, the heights $h_{J}^{\left(n\right)}$of
the three towers $Z_{J}^{\left(n\right)}$, $J\in\left\{ L,M,S\right\} ,$
at stage $n$ are given by $h_{J}^{\left(n\right)}=q_{n}$ for $J\in\left\{ M,L\right\} $
and $h_{S}^{\left(n\right)}=q_{n-1}$. By the structure of the first
return map $T^{\left(n\right)}$, the intervals $\left(T^{\left(n\right)}\right)^{i}\left(I_{S}^{\left(n\right)}\right)$,
$i=1,...,a_{n}$ partition the interval $\left[-\alpha^{\left(n-1\right)},-\alpha^{\left(n\right)}\right)=I^{\left(n\right)}\setminus I^{\left(n+1\right)}$
into intervals of equal length, and it follows that the first return
time of $T_{\alpha_{0}}$ is constant on $I_{S}^{\left(n\right)}$
and equals to 
\[
a_{n}\cdot h_{L}^{\left(n\right)}+h_{S}^{\left(n-1\right)}=a_{n}q_{n}+q_{n-1}=q_{n+1}.
\]

\begin{figure}[!tbh]
\includegraphics[width=0.6\textwidth]{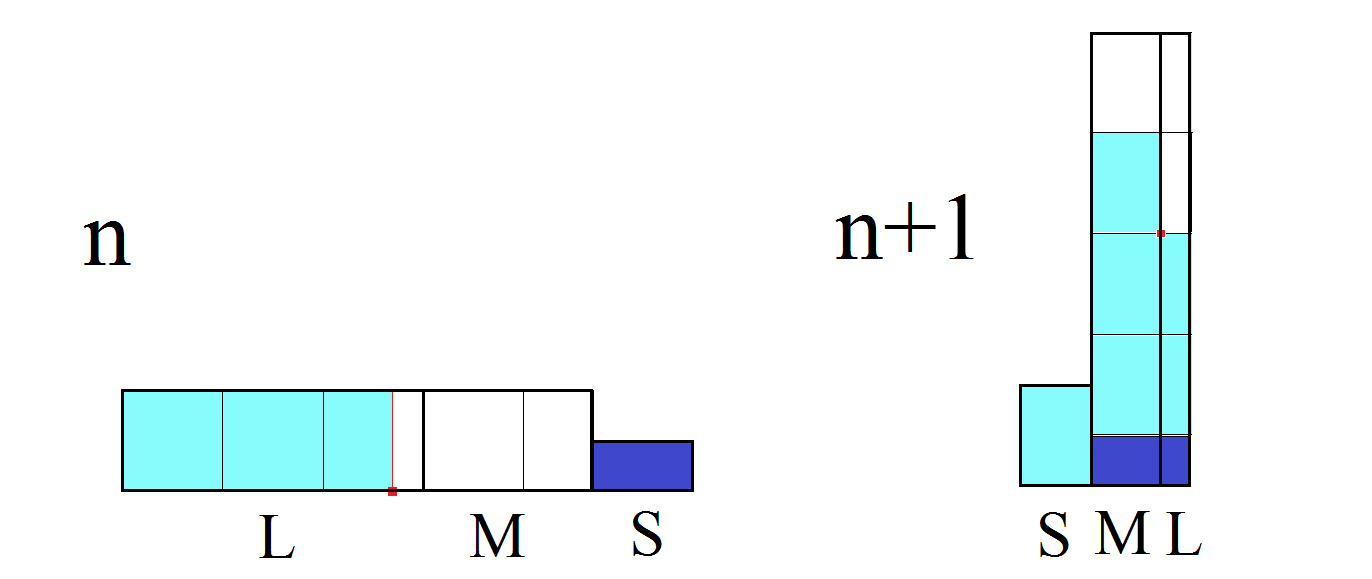}

\caption{\label{fig:The-tower-structure}The tower structure at step $n$ and
$n+1$ in the case when $\left(\alpha_{n},\beta_{n}\right)\in X_{G}.$
In this example $a_{n}=4$ and $b_{n}=3$. }
\end{figure}

It also follows that the tower over $I_{S}^{\left(n\right)}\subset I^{\left(n+1\right)}$
at stage $n+1$ is obtained by stacking the subtowers over the intervals
$\left(T^{\left(n\right)}\right)^{i}\left(I_{S}^{\left(n\right)}\right)$
on top of the tower $Z_{S}^{\left(n\right)}$ (as shown in Figure
\ref{fig:The-tower-structure}). By construction, the point $\beta^{\left(n+1\right)}$
is obtained by vertically projecting the point $\beta^{\left(n\right)}$
from its location in the tower over $I_{S}^{\left(n\right)}$ down
to the interval $I_{S}^{\left(n\right)}$. According to our definitions,
$\beta^{\left(n+1\right)}$ divides $I_{S}^{\left(n\right)}$ into
$I_{L}^{\left(n+1\right)}=\left[\beta^{\left(n\right)},\alpha^{\left(n-1\right)}\right)$
and $I_{M}^{\left(n+1\right)}=\left[0,\beta^{\left(n\right)}\right)$.
As we have seen, the height of the towers at stage $n+1$ over the
intervals $I_{M}^{\left(n+1\right)}$ and $I_{L}^{\left(n+1\right)}$
is the same and equals $q_{n+1}$, but the composition of the towers
is different. The tower $Z_{M}^{\left(n+1\right)}$ is obtained by
stacking, on top of the bottom tower $Z_{S}^{\left(n\right)}$, first
$b_{n}$ subtowers of $Z_{L}^{\left(n\right)}$ and then $a_{n}-b_{n}$
subtowers of $Z_{M}^{\left(n\right)}$ on top of them; $Z_{L}^{\left(n+1\right)}$
has a similar structure, with the tower $Z_{S}^{\left(n\right)}$
in the bottom, but with $b_{n}-1$ subtowers of $Z_{L}^{\left(n\right)}$
on top and then $a_{n}-b_{n}+1$ subtowers of $Z_{M}^{\left(n\right)}$
stacked over (see Figure \ref{fig:The-tower-structure}). The tower
over $I_{S}^{\left(n+1\right)}=I_{M}^{\left(n\right)}$ remains unchanged,
i.e $Z_{S}^{\left(n+1\right)}=Z_{M}^{\left(n\right)}$. 

\smallskip{}
It is convenient to describe the tower structure in the language of
\emph{substitutions}. Let us recall that a substitution $\tau$ on
a finite alphabet $\mathcal{A}$ is a map which associates to each
letter of $\mathcal{A}$ a finite word in the alphabet $\mathcal{A}$.
To each $(\alpha,\beta)$ with $\beta$ rational or $\alpha,\beta,1$
linearly independent over $\bbQ$, we associate a sequence $(\tau_{n})_{n}$
of substitutions over the alphabet $\{L,M,S\}$, where for $J\in\left\{ L,M,S\right\} $,
\[
\tau_{n}(J)=J_{0}J_{1}\cdots J_{k},\qquad\text{where\,}J,J_{0},\dots,J_{k}\in\{L,M,S\},
\]
if and only if the tower $Z_{J}^{\left(n+1\right)}$ consists of subtowers
of $Z_{J_{i}}^{\left(n\right)}$, $i=0,...,k$ stacked on top of each
other in the specified order, i.e. the subtower of $Z_{J_{i+1}}^{\left(n\right)}$
is stacked on top of $Z_{J_{i}}^{\left(n\right)}$. More formally,
\begin{equation}
\tau_{n}(J)=J_{0}J_{1}\cdots J_{k}\qquad\text{\ensuremath{\Leftrightarrow}\qquad}h_{J}^{\left(n+1\right)}=\sum_{j=0}^{k}h_{J_{j}}^{\left(n\right)}\quad and\quad(T^{(n)})^{i}(x)\in I_{J_{i}}^{(n)}\quad\forall x\in I_{J}^{(n+1)},\quad i=0,...,k.\label{eq:towerrecursivestructure}
\end{equation}

For example, in the case discussed above, since the tower $Z_{M}^{\left(n+1\right)}$
is obtained by stacking, on top of each other, in order, $Z_{S}^{\left(n\right)}$,
then $b_{n}$ subtowers of $Z_{L}^{\left(n\right)}$ and then $a_{n}-b_{n}$
subtowers of $Z_{M}^{\left(n\right)}$ , we have
\[
\tau_{n}(M)=S\underbrace{\,L\,\cdots\,L}_{b_{n}\text{ times}}\underbrace{M\,\cdots\,M}_{a_{n}-b_{n}\text{ times}}.
\]
We will use the convention of writing $J^{n}$ for the block $J\cdots J$
where the symbol $J$ is repeated $n$ times. With this convention,
the above substitution can be written $\tau_{n}(M)=SL^{b_{n}}M^{a_{n}-b_{n}}.$ 

If $\omega$ is a word $\omega=J_{0}J_{1}\cdots J_{k}$where we will
denote by $\omega_{i}$ the letter indexed by $0\leq i$<$\left|\omega\right|$.
Using this notation, we can rewrite (\ref{eq:towerrecursivestructure})
as
\[
h_{J}^{\left(n+1\right)}=\sum_{i=0}^{\left|\omega\right|-1}h_{\tau_{n}(J)_{i}}^{\left(n\right)}.
\]

\smallskip{}
We summarize the tower structure and the associated sequence of substitutions
in the following proposition. The substitution $\tau_{n}$ is determined
by the location of $\beta^{\left(n\right)}\in I^{\left(n\right)}$,
or equivalently, by the non-rescaled parameters $\left(\alpha_{n},\beta_{n}\right)$
and one can check that there are three separate cases corresponding
to the parameters being in $X_{G}$, $X_{B_{-}}$ or $X_{B_{+}}$.
One of the cases was analyzed in the discussion above, while the other
cases can be deduced similarly, and the proof of the proposition is
a straightforward induction on $n$. 
\begin{prop}
\label{prop: Tower Structure}The first return time function of $T_{\alpha_{0}}$
to $I^{\left(n\right)}$ is constant on each of the three intervals
$I_{J}^{\left(n\right)}$, $J\in\left\{ L,M,S\right\} $. Thus, for
$n=0,1,2,...$, 
\[
I^{\left(0\right)}=\bigcup_{J\in\left\{ L,M,S\right\} }Z_{J}^{\left(n\right)}\qquad\text{where\ \ensuremath{Z_{J}^{\left(n\right)}}=\ensuremath{\bigcup\limits _{i=0}^{h_{J}^{\left(n\right)}-1}T_{\alpha_{0}}^{i}I_{J}^{\left(n\right)}}}
\]
where $h_{J}^{\left(n\right)}$ is the value of the first return time
function on $I_{j}^{\left(n\right)}$, which is given by 
\[
h_{L}^{\left(n\right)}=q_{n},\quad h_{S}^{\left(n\right)}=q_{n-1},\quad h_{M}^{\left(n\right)}=\begin{cases}
q_{n} & if\ \beta_{n}\in\left[-1,0\right)\\
q_{n-1} & if\ \beta_{n}\in\left[0,\alpha_{n}\right)
\end{cases}.
\]
The sequence of substitutions associated to the pair $\left(\alpha,\beta\right)$
is given by the formulas, determined by the following cases
\begin{itemize}
\item If $\left(\alpha_{n},\beta_{n}\right)\in X_{G}$
\[
\begin{cases}
\tau_{n}\left(L\right)=SL^{b_{n}-1}M^{a_{n}-b_{n}+1}\\
\tau_{n}\left(M\right)=SL^{b_{n}}M^{a_{n}-b_{n}}\\
\tau_{n}\left(S\right)=M
\end{cases}
\]
\item If $\left(\alpha_{n},\beta_{n}\right)\in X_{B_{-}}$
\[
\begin{cases}
\tau_{n}\left(L\right)=SL^{a_{n}}\\
\tau_{n}\left(M\right)=M\\
\tau_{n}\left(S\right)=L
\end{cases}
\]
\item If $\left(\alpha_{n},\beta_{n}\right)\in X_{B_{+}}$
\[
\begin{cases}
\tau_{n}\left(L\right)=SL^{a_{n}}\\
\tau_{n}\left(M\right)=ML^{a_{n}}\\
\tau_{n}\left(S\right)=L
\end{cases}
\]
\end{itemize}
\end{prop}
\begin{rem}
\label{rem: Towers generate the sigma-field}It can be shown that
due to irrationality of $\alpha$, the levels of the towers $Z_{J}^{\left(n\right)}$,
$J\in\left\{ L,M,S\right\} $ form an increasing sequence of partitions
that separates points and hence generates the Borel $\sigma$-algebra
on $\left[-1,\alpha_{0}\right)$ (see for example \cite{sinai1994topics}).
\end{rem}
Let $A_{n}$, $n\in\mathbb{N}$, be the $3\times3$ \emph{incidence
matrix} of the substitution $\tau_{n}$ with entries indexed by $\{L,M,S\}$,
where the entry indexed by $\left(J_{1},J_{2}\right)$ , which we
will denote by $(A_{n})_{J_{1},J_{2}}$, gives the number of subtowers
contained in $Z_{J_{2}}^{(n)}$ among the subtowers of level $n$
which are stuck to form of tower $Z_{J_{1}}^{(n+1)}$. Equivalently,
the entry $(A_{n})_{J_{1},J_{2}}$ gives the number of occurrences
of the letter $J_{2}$ in the word $\tau_{n}\left(J_{1}\right)$.
If we adopt the convention that the order of rows/columns of $A_{n}$
corresponds to $L,M,S$, it follows from Proposition \ref{prop: Tower Structure}
that these matrices are then explicitly given by:

\begin{equation}
A_{n}=\begin{pmatrix}b_{n}-1 & a_{n}-b_{n}+1 & 1\\
b_{n} & a_{n}-b_{n} & 1\\
0 & 1 & 0
\end{pmatrix}\qquad\text{if}\,\left(\alpha_{n},\beta_{n}\right)\in G,\label{eq: Transition mat 1}
\end{equation}

\begin{equation}
A_{n}=\begin{pmatrix}a_{n} & 0 & 1\\
0 & 1 & 0\\
1 & 0 & 0
\end{pmatrix}\qquad\text{if}\ \left(\alpha_{n},\beta_{n}\right)\in B_{-},\label{eq: Transition mat 2}
\end{equation}

\begin{equation}
A_{n}=\begin{pmatrix}a_{n} & 0 & 1\\
a_{n} & 1 & 0\\
1 & 0 & 0
\end{pmatrix}\qquad\text{if}\ \left(\alpha_{n},\beta_{n}\right)\in B_{+}.\label{eq: Transition mat 3}
\end{equation}
In particular, if we denote by $h^{(n)}$ the column vector of heights
towers, i.e. the transpose of $\left(h_{L}^{(n)},h_{M}^{(n)},h_{S}^{(n)}\right)$,
it satisfies the recursive relations
\begin{equation}
h^{(n+1)}=A_{n}h^{(n)},\qquad n\in\mathbb{N\cup\text{\ensuremath{\left\{  0\right\} } .}}\label{eq:recursiveheights}
\end{equation}

\begin{rem}
\label{rem:Vershik map}We remark briefly for the readers familiar
with the \emph{Vershik adic map} and the $S$-adic formalism (even
though it will play no role in the rest of this paper), that the sequence
$(\tau_{n})_{n}$ also allows to represent the map $T_{\alpha_{0}}$
as a \emph{Vershik adic map}. The associated Bratteli diagram is a
non-stationary diagram, whose vertex sets $V_{n}$ are always indexed
by $\{L,M,S\}$ with $(A_{n})_{J_{1},J_{2}}$ edges from $J_{1}$
to $J_{k}$; the ordering of the edges which enter the vertex $J$
at level $n$ is given exactly by the substitution word $\tau_{n}(J)$.
We refer the interested reader to the works by Vershik \cite{vershik1997adic}
and to the  paper by Berthe and Delecroix \cite{berthe2013beyond}
for further information on Vershik maps, Brattelli diagrams and $S$-adic
formalism. 
\end{rem}

\subsubsection{Special Birkhoff sums\label{subsec:Special-Birkhoff-sums}}

Let us consider now the function $\varphi$ defined by (\ref{eq: Cocycle function})
which has a discontinuity at $0$ and at $\beta_{0}$. In order to
study its Birkhoff sums $\varphi_{n}$ (defined in (\ref{eq:defBS}))
, we will use the renormalization algorithm described in the previous
section. Under the assumption that $\left(\alpha,\beta\right)\in\tilde{X}$,
$\varphi$ determines a sequence of functions $\varphi^{\left(n\right)}$,
where $\varphi^{(n)}$ is a real valued function defined on $I^{(n)}$
obtained by \emph{inducing} $\varphi$ on $I^{(n)}$, i.e. by setting
\begin{equation}
\varphi^{(n)}(x)=\sum_{i=0}^{h_{J}^{(n)}-1}\varphi({T_{\alpha}}^{i}(x)),\qquad\mathrm{if\,x}\in I_{J}^{(n)}.\label{eq:special BS}
\end{equation}
The function $\varphi^{(n)}$ is what Marmi-Moussa-Yoccoz in \cite{marmi2005cohomological}
started calling \emph{special Birkhoff sum}s: the value $\varphi^{(n)}(x)$
gives the Birkhoff sum of the function $\varphi$ along the orbit
of $x\in I_{J}^{(n)}$ until its first return to $I^{(n)}$, i.e.
it represents the Birkhoff sum of the function along an orbit which
goes from the bottom to the top of the tower $Z_{J}^{(n)}$.

One can see that since $\varphi^{(0)}:=\varphi$ has mean zero and
a discontinuity with a jump of $1$ at $\beta^{(0)}:=\beta_{0}$,
its special Birkhoff sums $\varphi^{(n)}$, $n\in\mathbb{N}$, again
have mean zero and a discontinuity with a jump of $1$. The points
$\beta^{(n)}$, $n\in\mathbb{N}$, are defined in the renormalization
procedure exactly so that $\varphi^{(n)}$ has a jump of one at $\beta^{(n)}$.
Moreover, the function $\varphi$ is constant on each level of the
towers $Z_{J}^{\left(n\right)}$, $J\in\left\{ L,M,S\right\} $, $n\in\mathbb{N}\cup\left\{ 0\right\} $,
and therefore, it is completely determined by a sequence of vectors
\[
\left(\varphi_{L}^{\left(n\right)},\varphi_{M}^{\left(n\right)},\varphi_{S}^{\left(n\right)}\right),\qquad n\in\mathbb{N}\cup{0},
\]
where $\varphi_{J}^{\left(n\right)}=\varphi^{\left(n\right)}\left(x\right)$,
for any $x\in I_{J}^{\left(n\right)}$. It then follows immediately
from the towers recursive structure (see equation (\ref{eq:towerrecursivestructure}))
that the functions $\varphi^{\left(n\right)}$ also satisfy the following
recursive formulas given by the substitutions in Proposition \ref{prop: Tower Structure}:

\[
\varphi_{J}^{\left(n+1\right)}=\sum_{i=0}^{k}\varphi_{J_{i}}^{\left(n\right)}\qquad\text{if \,\ensuremath{\tau}}_{n}(J)=J_{0}\cdots J_{k}.
\]

We finish this section with a few simple observations on the heights
of the towers and on special Birkhoff sums along these towers that
we will need for the proof of the main result. Let $\left(\alpha_{0},\beta_{0}\right)\in\tilde{X}$
be the parameters associated to a given pair $\left(\alpha,\beta\right)$
via the relations (\ref{eq:Correcpondence between alpha and alpha_zero})
and (\ref{eq: Correcspondence between beta and beta_zero}). Under
the assumption that $\alpha$ is badly approximable, since the heights
of the towers appearing in the renormalization procedure satisfy $\eqref{eq:recursiveheights}$)
and $0\leq b_{n}\leq a_{n}$ are bounded, there exists a constant
$C$ such that

\begin{equation}
C^{-1}n\leq\log h_{J}^{\left(n\right)}\leq Cn\qquad\textrm{for any\,}n\in\mathbb{N}.\label{eq: Heights of towers growth}
\end{equation}
It follows that for any $m\in\bbN$, there exists a constant $M=M\left(m\right)$,
such that if $\left|k-n\right|\leq m$, then 
\begin{equation}
\frac{1}{M}<\frac{h_{J}^{\left(n\right)}}{h_{K}^{\left(k\right)}}<M\qquad\textrm{for any\,}J,K\in\left\{ L,M,S\right\} .\label{eq: Inequality for heights of towers}
\end{equation}
Moreover, by (\ref{eq: DK inequality}), the special Birkhoff $\varphi_{J}^{\left(n\right)}$
are uniformly bounded, i.e. 
\begin{equation}
\sup\left\{ \left|\varphi_{J}^{\left(n\right)}\right|:\ J\in\left\{ L,M,S\right\} ,\ n\in\bbN\right\} <\infty.\label{eq: Denjoy Koksma}
\end{equation}

\subsection{The (adic) symbolic coding\label{subsec: Symbolic Coding}}

The renormalization algorithm and the formalism defined above lead
to the symbolic coding of the dynamics of $T_{\alpha_{0}}$ described
in the present section. This coding is exploited in Section \ref{subsec: Markov chain towers}
to build an array of non-homogeneous Markov chains which models the
dynamics. 
\begin{defn}
(Markov compactum) Let $\left(\mathcal{S}_{n}\right)_{n=1}^{\infty}$
be a sequence of finite sets with $\sup_{i}\left|\mathcal{S}_{i}\right|<\infty$
and let $\left(A^{\left(n\right)}\right)_{n=1}^{\infty}$ be a sequence
of matrices, such that $A^{\left(n\right)}$ is an $\left|\mathcal{S}_{n}\right|\times\left|\mathcal{S}_{n+1}\right|$
matrix whose entries $A_{s,t}^{\left(n\right)}\in\left\{ 0,1\right\} $
for any $(s,t)\in\mathcal{S}_{n}\times\mathcal{S}_{n+1}$. The \emph{Markov
compactum} determined by $A^{\left(n\right)}$ is the space 
\[
X=\left\{ \omega\in\prod_{n=1}^{\infty}\mathcal{S}_{n}:\ A_{s_{n},s_{n+1}}^{(n)}=1\text{ for all }n\in\bbN\right\} .
\]

\begin{figure}[!tbh]

\includegraphics[angle=-90,width=0.8\textwidth]{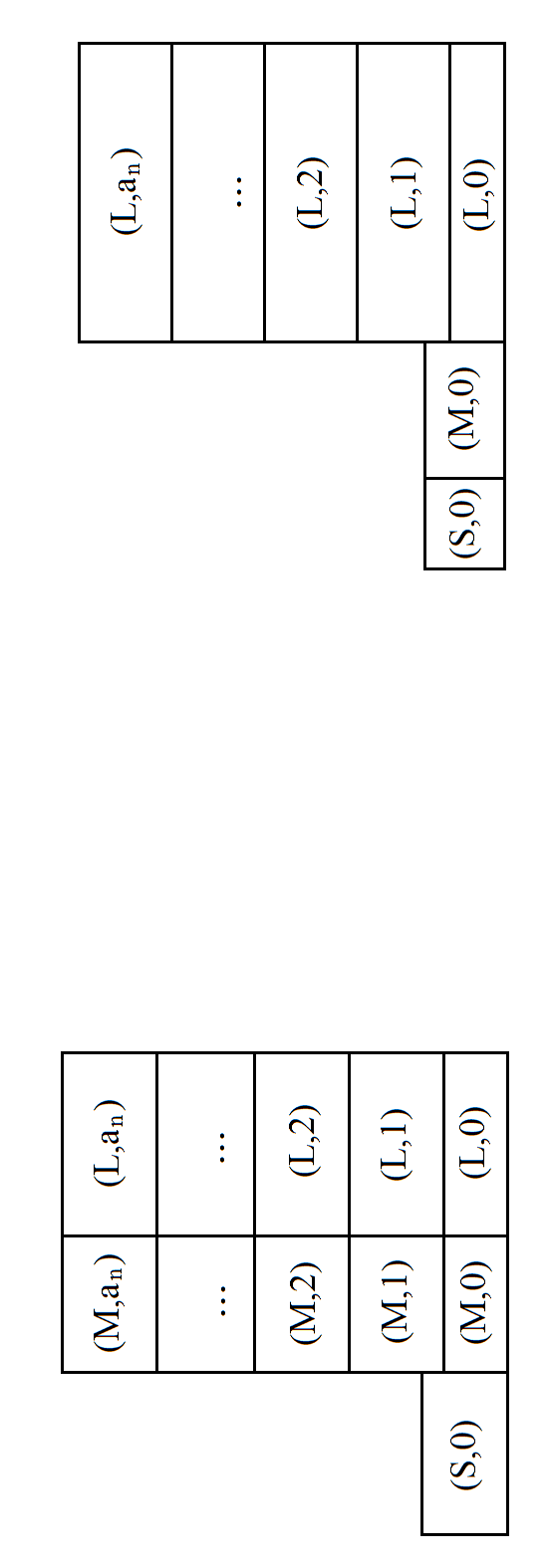}\caption{\label{fig:Labeling-of-the}Labeling of the subtowers of the towers$Z_{J}^{(n)}$
by labels $(J,i$), $J\in\left\{ L,M,S\right\} $. }
\end{figure}

To describe the coding, recall that for each $n\in\mathbb{N}$, each
tower $Z_{J}^{(n)}$, where $J\in\{L,M,S\}$, is obtained by stacking
at most $a_{n}+1$ \emph{subtowers }of the towers $Z_{K}^{(n-1)}$
(the type and order of the subtowers is completely determined by the
word $\tau_{n-1}(J)$ given by the substitution $\tau_{n-1}$ as described
in Proposition \ref{prop: Tower Structure}). We will \emph{label}
these subtowers by $(J,i)$, where the index $i$ satisfies $0\leq i\leq a_{n}$
and indexes the subtowers \emph{from bottom to top}: more formally,
$(J,i)$ is the label of the subtower of $Z_{J}^{(n)}$, with base
$(T^{(n-1)})^{i}(I_{J}^{(n)})$, which is the $(i+1)^{th}$ subtower
from the bottom (see Figure \ref{fig:Labeling-of-the}). Thus, for
a fixed $n$, denoting by $\left|\tau_{n-1}\left(J\right)\right|$
the length of the word $\tau_{n-1}\left(J\right)$, the labels of
the subtowers belong to
\[
\left\{ \left(J,i\right):\quad J\in\left\{ L,M,S\right\} ,\ i=0,...,\left|\tau_{n-1}\left(J\right)\right|-1\right\} .
\]
When $\alpha=\frac{\alpha_{0}}{1+\alpha_{0}}=\left[a_{0},a_{1},\dots,a_{n},\dots\right]$
is the continued badly approximable, let $a_{max}$ be the largest
of its continued fraction entries and consider the alphabet 
\[
E=E(a_{max})=\{L,M,S\}\times\{0,\cdots,a_{max}\}.
\]
\end{defn}
\begin{rem}
It is not necessary for $\alpha$ to be badly approximable in order
for the construction of the present section and the next section to
be valid. If $\alpha$ is not badly approximable, define $E=\left\{ L,M,S\right\} \times\left\{ 0,1,...,n,...\right\} $
. This definition would make all statements of this and the following
sections valid, without any further changes. 
\end{rem}
\begin{defn}
Given $x\in\left[-1,\alpha_{0}\right)$, for each $n\in\mathbb{N}$,
$x$ is contained in a unique tower $Z_{J_{n}(x)}^{(n)}$ for some
$J_{n}\left(x\right)\in\{L,M,S\}$, and furthermore in a unique subtower
of stage $n-1$ inside it, labeled by $\left(J_{n}\left(x\right),j_{n}\left(x\right)\right)$
where $0\leq j_{n}\left(x\right)\leq a_{n}$. Let $\Psi:\left[-1,\alpha_{0}\right)\rightarrow E$
be the coding map defined by 
\[
\Psi(x):=\left(J_{n}\left(x\right),j_{n}\left(x\right)\right)\in E^{\mathbb{N}}.
\]
\end{defn}
Let us recall that for word $\omega$ in the alphabet $E$ let us
denote by $\omega_{i}$ the letter in the word which is labeled by
$0\leq i<$$\left|\omega\right|$ . 
\begin{prop}
\label{prop: Coding}The image of $\Psi$ is contained in the subspace
$\Sigma\subset E^{\mathbb{N}}$ defined by 
\[
\Sigma:=\left\{ \left(\left(J_{1},j_{1}\right),...,\left(J_{n},j_{n}\right),...\right)\in E^{\bbN}:\ \left(\tau_{i}\left(J_{i+1}\right)\right)_{j_{i+1}}=J_{i},\ i=1,2,...\right\} .
\]
The preimage under $\Psi$ of any cylinder $\left[\left(J_{1},j_{1}\right),...,\left(J_{n},j_{n}\right)\right]:=\left\{ \omega\in\Sigma:\ \omega_{i}=\left(J_{i},j_{i}\right),\ i=1,...,n\right\} $
satisfying the constraints $\left(\tau_{i}\left(J_{i+1}\right)\right)_{j_{i+1}}=J_{i},\ i=1,2,...,n-1$
is the set of all points on some level of the tower $Z_{J_{n}}^{\left(n\right)}$,
i.e. there exists $0\leq i<h_{J_{n}}^{\left(n\right)}$ such that
\[
\Psi^{-1}\left(\left[\left(J_{1},j_{1}\right),...,\left(J_{n},j_{n}\right)\right]\right)=T_{\alpha_{0}}^{i}\left(I_{J_{n}}^{\left(n\right)}\right).
\]
Moreover, $\Psi$ is a Borel isomorphism between $\left[-1,\alpha_{0}\right)$
and its image, where the Borel structure on the image of $\Psi$ is
inherited from the natural Borel structure on $E^{\bbN}$ arising
from the product topology on $E^{\bbN}$. 
\end{prop}
Let 
\begin{equation}
\mathcal{S}_{n}:=\left\{ \left(J,j\right)\in E:\ \exists\omega\in\Sigma\ s.t.\ \omega_{n}=\left(K,k\right)\right\} \label{eq: State space of MC}
\end{equation}
be the set of symbols which appear as $n^{th}$ coordinate in some
admissible word in $\Sigma$, and note that the definition of $\Sigma$
shows that $\Sigma$ is a Markov compactum with state space $\prod_{i=1}^{\infty}\mathcal{S}_{i}$,
given by a sequence of matrices $\left(A^{(n)}\right)^{T}$ indexed
by $\mathcal{S}_{n}\times\mathcal{S}_{n+1}$ such that $A_{\left(K,k\right),\left(J,j\right)}^{\left(n\right)}=1$
if and only if $\left(\tau_{n}\left(J\right)\right)_{j}=K$. Although
we do not need it in what follows, one can explicitly describe the
image $\Sigma'\subset\Sigma$ of the coding map $\Psi$ and show that
it is obtained from $\Sigma$ by removing countably many sequences.
We remarked in Remark \ref{rem:Vershik map} that $T_{\alpha}$ is
conjugated to a the Vershik adic map. Let us add that the map $\Psi$
provides the measure theoretical conjugacy.
\begin{proof}[Proof of Proposition \ref{prop: Coding}]
First we prove that the image of $\Psi$ is contained in $\Sigma$.
To see this, note that for $x\in\left[-1,\alpha_{0}\right)$, $(J_{n}\left(x\right),j_{n}\left(x\right))=(K,k)$
means that $x$ belongs to $Z_{K}^{(n)}$ (since $J_{n}\left(x\right)=K$
) and $(J_{n+1}\left(x\right),j_{n+1}\left(x\right))=(J,j)$ means
that $x$ belongs to the $j^{th}$ subtower of $Z_{J}^{(n+1)}$. Hence
the $j^{th}$ subtower of $Z_{J}^{(n+1)}$ must be contained in $Z_{K}^{(n)}$
. Recalling the definition of the substitutions $(\tau_{n})_{n},$
this implies exactly the relation $\left(\tau_{n}\left(J\right)\right)_{j}=K$,
which in turn implies that $\Psi\left(x\right)\in\Sigma$. 

To prove the second statement, namely that cylinders correspond to
floors of towers, note that according to our labeling of the towers,
the set of all $x$ such that $\left(J_{n}\left(x\right),j_{n}\left(x\right)\right)=\left(J_{n},j_{n}\right)$
consists exactly of all points which belong to $h_{K}^{\left(n-1\right)}$
levels of the tower $Z_{J_{n}}^{\left(n\right)}$, where $K=\left(\tau_{n-1}\left(J_{n}\right)\right)_{j_{n}}$.
Proceeding by induction, one sees that for any $k=n-1,\dots,1$, the
set $\left\{ x:\ \left(J_{i}\left(x\right),j_{i}\left(x\right)\right)=\left(J_{i},j_{i}\right),\ i=k,...,n\right\} $
is the set of all points contained in precisely $h_{K}^{\left(k-1\right)}$
levels of the tower $Z_{J_{n}}^{\left(n\right)}$, where $K=\left(\tau_{k-1}\left(J_{k}\right)\right)_{j_{k}}$.
Thus, since $h_{K}^{(0)}=1$ for any $K\in\left\{ L,M,S\right\} $,
$\Psi^{-1}\left(\left[\left(J_{1},j_{1}\right),...,\left(J_{n},j_{n}\right)\right]\right)$
is the set of all points on a \emph{single} level of the tower $Z_{J_{n}}^{\left(n\right)}$.
This argument shows that the levels of the towers $Z_{J}^{\left(n\right)}$,
$J\in\left\{ L,M,S\right\} ,$ are in bijective correspondence under
the map $\Psi$ with cylinders of length $n$ in $\Sigma$. 

Finally, injectivity and bi-measurability of $\Psi$ follow since
the sequence of partitions induced by the tower structure generate
the Borel sets of the space $\left[-1,\alpha_{0}\right)$ and separates
points (see Remark \ref{rem: Towers generate the sigma-field}). 
\end{proof}

\subsection{The Markov chain modeling towers. \label{subsec: Markov chain towers}}

In what follows, we denote by $\mu$ be the push forward by the map
$\Psi$ of the normalized Lebesgue measure $\lambda$ on $\left[-1,\alpha_{0}\right)$,
i.e. the measure given by
\begin{equation}
\mu:=\lambda\circ\Psi^{-1}.\label{eq:def mu}
\end{equation}
Moreover, for $J\in\left\{ L,M,S\right\} $ , let us define also the
conditional measures
\begin{equation}
\mu_{n}^{J}\left(\left[\left(J_{1},j_{1}\right),...,\left(J_{n},j_{n}\right)\right]\right):=\mu\left(\left[\left(J_{1},j_{1}\right),...,\left(J_{n},j_{n}\right)\right]\left|J_{n}=J\right.\right).\label{eq:def mu J}
\end{equation}
We denote by $\Sigma_{n}$ and $E_{n}$ correspondingly, the restriction
of $\Sigma$ and $E$ to the first $n$ coordinates and we endow these
sets with the $\sigma$-algebras inherited from the Borel $\sigma$-algebra
on $E^{\bbN}$. Let $\mathcal{S}_{n}$ be the set of states appearing
in the $n^{th}$ coordinate of $\Sigma$, defined by (\ref{eq: State space of MC}). 

We define a sequence of transition probabilities, or equivalently
in this discrete case, stochastic matrices $p_{\left(J,j\right),\left(K,k\right)}^{\left(n\right)}$,
where $\left(J,j\right)\in\mathcal{S}_{n+1}$ and $\left(K,k\right)\in\mathcal{S}_{n}$,
and a sequence of probability distributions $\pi_{n}$ on $\mathcal{S}_{n}$
which are used to define a sequence of Markovian measures on $\Sigma_{n}$
that model the dynamical renormalization procedure. We refer to the
sequence $p^{\left(n\right)}$ as the sequence of\textit{ transition
matrices associated to the pair} $\left(\alpha_{0},\beta_{0}\right)\in\tilde{X}$. 
\begin{defn}
\label{def: Markov transitions}For any $n\in\bbN$ and $J,\,K\in\left\{ L,M,S\right\} $,
if 
\[
\tau_{n}\left(J\right)=J_{0}\dots J_{l}\text{\,\ and\,\ }\tau_{n-1}\left(K\right)=K_{0}\dots K_{m},
\]
we define 
\end{defn}
\[
p_{\left(J,j\right),\left(K,k\right)}^{\left(n\right)}:=\begin{cases}
\frac{h_{K_{k}}^{\left(n-1\right)}}{h_{K}^{\left(n\right)}} & \text{if}\ K=J_{j}\ \mathrm{and}\ 0\leq j\leq l,\ k\leq m,\\
0 & \text{otherwise; }
\end{cases}
\]
\[
\pi_{n}\left(K,k\right):=\begin{cases}
\lambda\left(Z_{K}^{\left(n\right)}\right)\cdot\frac{h_{K_{k}}^{\left(n-1\right)}}{h_{K}^{\left(n\right)}} & \text{if}\ 0\leq k\leq m,\\
0 & \text{otherwise.}
\end{cases}
\]
Moreover, for any $\left(L,l\right)\in E$, we set 

\[
\pi_{n}^{K}\left(L,l\right):=\begin{cases}
\frac{h_{K_{l}}^{\left(n-1\right)}}{h_{K}^{\left(n\right)}} & \text{if}\ 0\leq l\leq m\ \mathrm{and}\ L=K,\\
0 & \text{otherwise.}
\end{cases}.
\]

\begin{rem}
The rationale behind the definition of $\pi_{n}$ is that $\pi_{n}\left(K,k\right)$
is defined to be the $\lambda$- measure of the piece of the tower
$Z_{K}^{(n)}$ labeled by $(K,k)$; similarly $p_{\left(J,j\right),\left(K,k\right)}^{\left(n\right)}$
is non zero exactly when the $k^{th}$ subtower inside $Z_{J}^{(n)}$
is contained in $Z_{K}^{(n-1)}$, in which case it gives the proportion
of this $k^{th}$ subtower which is contained in the subtower of $Z_{K}^{(n-1)}$
labeled by $(J,j)$.
\end{rem}
The following Proposition identifies the measures $\mu_{n}$ and $\mu_{n}^{J}$
as Markovian measures on $\Sigma_{n}$ generated by the transition
matrices and initial distributions indicated in the previous definition.
\begin{prop}
\label{prop: Markovian measure}For every $n\in\bbN$, $J\in\left\{ L,M,S\right\} $
and every word $\left(\left(J_{1},j_{1}\right),...,\left(J_{n},j_{n}\right)\right)\in\Sigma_{n}$
we have
\begin{equation}
\mu\left(\left[\left(J_{1},j_{1}\right),...,\left(J_{n},j_{n}\right)\right]\right)=\pi_{n}\left(J_{n},j_{n}\right)\prod_{i=1}^{n-1}p_{\left(J_{i+1},j_{i+1}\right),\left(J_{i},j_{i}\right)}^{\left(i\right)}\label{eq: Correspondence between Markov meas and Leb}
\end{equation}
and 
\begin{equation}
\mu\left(\left[\left(J_{1},j_{1}\right),...,\left(J_{n},j_{n}\right)\right]\left|J_{n}=J\right.\right)=\pi_{n}^{J}\left(J_{n},j_{n}\right)\prod_{i=1}^{n-1}p_{\left(J_{i+1},j_{i+1}\right),\left(J_{i},j_{i}\right)}^{\left(i\right)}.\label{eq: Correspondence between Markov meas and Leb 2}
\end{equation}
\end{prop}
\begin{proof}
By Proposition \ref{prop: Coding}, $\Psi^{-1}\left(\left[\left(J_{1},j_{1}\right),...,\left(J_{n},j_{n}\right)\right]\right)$
is non empty if and only if the sequence $\left(J_{1},j_{1}\right),...,\left(J_{n},j_{n}\right)$
satisfies the conditions $\left(\tau_{i}\left(J_{i+1}\right)\right)_{j_{i+1}}=J_{i}$
for $i=1,...,n-1$, in which case it consists of the the set of all
points on a certain level of the tower $Z_{J_{n}}^{\left(n\right)}$,
i.e. 
\[
\Psi^{-1}\left(\left[\left(J_{1},j_{1}\right),...,\left(J_{n},j_{n}\right)\right]\right)=T_{\alpha_{0}}^{1}\left(I_{J_{m}}^{(n)}\right)\quad\text{for some }0\leq l\leq h_{J_{n}}^{\left(n\right)}.
\]
It follows from the definition (\ref{eq:def mu}) of the measure $\mu$
that
\[
\mu\left(\left[\left(J_{1},j_{1}\right),...,\left(J_{n},j_{n}\right)\right]\right)=\lambda\left(I_{J_{n}}^{\left(n\right)}\right)=\frac{\lambda\left(Z_{J_{n}}^{\left(n\right)}\right)}{h_{J_{n}}^{\left(n\right)}}.
\]
Moreover, we get that the conditional measures $\mu_{n}^{J}$ given
by (\ref{eq:def mu J}) satisfy
\[
\mu_{n}^{J}\left(\left[\left(J_{1},j_{1}\right),...,\left(J_{n},j_{n}\right)\right]\right)=\mu\left(\left[\left(J_{1},j_{1}\right),...,\left(J_{n},j_{n}\right)\right]\left|J_{n}=J\right.\right)=\frac{1}{h_{J_{n}}^{\left(n\right)}}.
\]
Equations (\ref{eq: Correspondence between Markov meas and Leb})
now follow by definition of $\pi_{n}$ and $p^{\left(n\right)},$
which give that
\[
p_{\left(J_{i+1},j_{i+1}\right),\left(J_{i},j_{i}\right)}^{\left(i\right)}=\frac{h_{\tau_{i-1}(J_{i})_{j_{i}}}^{\left(i-1\right)}}{h_{\tau_{i}(J_{i+1})_{j_{i+1}}}^{\left(i\right)}},\qquad\pi_{n}\left(J_{n},j_{n}\right)=\lambda\left(Z_{J_{n}}^{\left(n\right)}\right)\cdot\frac{h_{\tau_{n-1}(J_{n})_{j_{n}}}^{\left(n-1\right)}}{h_{J_{n}}^{\left(n\right)}}.
\]
Hence, by the conditions $\left(\tau_{i}\left(J_{i+1}\right)\right)_{j_{i+1}}=J_{i}$
for $i=1,...,n-1$ and recalling that $J_{n}=J$ and $h_{K}^{(0)}=1$
for any $K\in\left\{ L,M,S\right\} $, we have that

\[
\pi_{n}\left(J_{n},j_{n}\right)\prod_{i=1}^{n-1}p_{\left(J_{i+1},j_{i+1}\right),\left(J_{i},j_{i}\right)}^{\left(i\right)}=\lambda\left(Z_{J_{n}}^{\left(n\right)}\right)\cdot\frac{h_{J_{n-1}}^{\left(n-1\right)}}{h_{J_{n}}^{\left(n\right)}}\left(\prod_{i=2}^{n-1}\frac{h_{J_{i-1}}^{\left(i-1\right)}}{h_{J_{i}}^{\left(i\right)}}\right)\frac{1}{h_{J_{1}}^{\left(1\right)}}=\lambda\left(Z_{J}^{\left(n\right)}\right).
\]
Equations (\ref{eq: Correspondence between Markov meas and Leb 2})
follows in the same way by using the definition of $\pi_{n}^{J}$
instead than $\pi_{n}$ .

Finally, if the sequence $\left(J_{1},j_{1}\right),...,\left(J_{n},j_{n}\right)$
does not satisfy the conditions $\left(\tau_{i}\left(J_{i+1}\right)\right)_{j_{i+1}}=J_{i}$
for $i=1,...,n-1$,$\Psi^{-1}\left(\left[\left(J_{1},j_{1}\right),...,\left(J_{n},j_{n}\right)\right]\right)=\emptyset$
(by Proposition \ref{prop: Coding}, as recalled above) and by definition
of $p^{\left(n\right)}$ and $\pi_{n}$, $\pi_{n}^{J}$, we get that
the right hand sides in (\ref{eq:Correcpondence between alpha and alpha_zero})
and (\ref{eq:Correspondence between sums on Markov chain and Birkhoff sums})
are both zero, so equations (\ref{eq:Correcpondence between alpha and alpha_zero})
and (\ref{eq:Correspondence between sums on Markov chain and Birkhoff sums})
hold in this case too. This completes the proof. 
\end{proof}
For $\omega\in\Sigma$, $n\in\bbN$, we define the \emph{coordinate
random variables }
\begin{equation}
X_{n}\left(\omega\right):=\omega_{n}\label{eq: Constructed Random variables}
\end{equation}
Since, all cylinders of the form $\left[\left(J_{1},j_{1}\right),...,\left(J_{n},j_{n}\right)\right]$,
with $\left(\left(J_{1},j_{1}\right),...,\left(J_{n},j_{n}\right)\right)\in\Sigma_{n}$
generate the $\sigma$-algebra of $\Sigma_{n}$, it immediately follows
from Proposition \ref{prop: Markovian measure} that for every $n\in\bbN$,
$X_{n},...,X_{1}$ form a Markov chain on $\Sigma_{n}$ with respect
to the measures $\mu$, $\mu_{n}^{J}$ with transition probabilities
$p^{\left(i\right)}$, $i=1,...,n-1$ and initial distributions $\pi_{n}$
, $\pi_{n}^{J}$, respectively. 

\subsection{The functions over the Markov chain modeling the Birkhoff sums. \label{subsec: Markov chain-1}}

Let us now define a sequence of functions $\xi_{n}:\,\mathcal{S}_{n}\rightarrow\bbR$
that enables us to model the distribution of Birkhoff sums. In section
\ref{subsec:Special-Birkhoff-sums} we introduced the notion of special
Birkhoff sums of $\varphi$, i.e. Birkhoff sums of $\varphi$ along
the orbit of a point $x$ in the base of a renormalization tower $Z_{J}^{(n)}$
up to the height of the tower, see (\ref{eq:special BS}). We will
consider in this section \emph{intermediate Birkhoff sums along a
tower} (of for short, \emph{intermediate Birkhoff sums}), namely Birkhoff
sums of a point at the base of a tower $Z_{J}^{\left(n\right)}$ up
to an intermediate height, i.e. sums of the form 
\[
\sum_{k=0}^{j}\varphi\left(T_{\alpha_{0}}^{k}\left(x\right)\right),\qquad\text{where}\quad x\in I_{J}^{(n)},\ 0\leq j<h_{J}^{(n)},\qquad n\in\mathbb{N}\cup\left\{ 0\right\} ,\ J\in\left\{ L,M,S\right\} .
\]

The crucial Proposition (\ref{prop: Connection between temporal sums and MC})
shows that intermediate Birkhoff sums can be expressed as sums of
the following functions $\left\{ \xi_{n}\right\} $ over the Markov
chain $\left(X_{n}\right)$. 
\begin{defn}
For $n\in\bbN$, $\left(J,j\right)\in \mathcal{S}_{n}$ such that $\tau_{n-1}\left(J\right)=J_{0}\dots J_{l}$
(note that this forces $0\leq j\leq l=\left|\tau_{n-1}\left(J\right)\right|$),
if $n\geq2$ set \label{def: xi functions}
\[
\xi_{n}\left(\left(J,j\right)\right)=\sum_{i=0}^{j-1}\varphi_{J_{i}}^{\left(n-1\right)}
\]
where, by convention, a sum with $i$ that runs from $0$ to $-1$
is equal to zero. If $n=1$, set 
\[
\xi_{1}\left(\left(J,j\right)\right)=\sum_{i=0}^{j}\varphi_{J_{i}}^{\left(0\right)}.
\]
We then have the following proposition.
\end{defn}
\begin{prop}
\label{prop: Connection between temporal sums and MC}Let $J\in\left\{ L,M,S\right\} $
and let ,$x_{J}\in I_{J}^{\left(n\right)}$. Then for any $A\in\mathcal{B}\left(\bbR\right)$,
\[
\frac{1}{h_{J}^{\left(n\right)}}\#\left\{ 0\leq j\leq h_{J}^{\left(n\right)}-1:\ \sum_{k=0}^{j}\varphi\left(T_{\alpha_{0}}^{k}\left(x_{J}\right)\right)\in A\right\} =\mu_{n}^{J}\left(\sum_{k=1}^{n}\xi_{k}\left(X_{k}\right)\in A\right).
\]
\end{prop}
\begin{proof}
We show by induction on $n$ that, for any $J\in\left\{ L,M,S\right\} $,
any $x\in I_{J}^{\left(n\right)}$ and $0\leq l\leq h_{J}^{\left(n\right)}-1$,
we have that
\begin{equation}
\sum_{k=0}^{l}\varphi\left(T_{\alpha_{0}}^{k}\left(x\right)\right)\in A\iff\sum_{k=1}^{n}\xi_{k}\left(\left(J_{k},j_{k}\right)\right)\in A.\label{eq:Correspondence between sums on Markov chain and Birkhoff sums}
\end{equation}
where $\left[\omega\right]=\left[\left(J_{1},j_{1}\right),...,\left(J_{n},j_{n}\right)\right]$
is the (unique) cylinder containing $T_{\alpha_{0}}^{l}x$. To see
this, note first that for $n=1$, $J\in\left\{ L,M,S\right\} $, $x\in I_{J}^{\left(1\right)}$
and $0\leq l\leq h_{J}^{\left(1\right)}-1$, we have $T^{l}x\in\Psi^{-1}\left(\left[\left(J,l\right)\right]\right)$
and by definition of $\xi_{1}$,
\[
\xi_{1}\left(\left(J,l\right)\right)=\sum_{k=1}^{l}\varphi_{\left(\tau_{0}\left(J\right)\right)_{k}}^{\left(0\right)}=\sum_{k=0}^{l-1}\varphi\left(T_{\alpha_{0}}^{k}\left(x\right)\right),
\]
which proves the claim for $n=1$. Now, assume that (\ref{eq:Correspondence between sums on Markov chain and Birkhoff sums})
holds for some $n\in\bbN$. Then for $J\in\left\{ L,M,S\right\} $,
$x\in I_{J}^{\left(n+1\right)}$, and $0\leq l\leq h_{J}^{\left(n+1\right)}-1$,
let $\left[\omega\right]=\left[\left(J_{1},j_{1}\right),...,\left(J_{n+1},j_{n+1}\right)\right]$
be the unique cylinder such that $T_{\alpha_{0}}^{l}\left(x\right)\in\Psi^{-1}\left(\left[\omega\right]\right)$.
Then $J_{n}=J$ and by Proposition \ref{prop: Coding} $\Psi^{-1}\left(\left[\omega\right]\right)=T_{\alpha_{0}}^{l}\left(I_{J}\right)$.
It follows from definition of the map $\Psi$, that $J_{n+1}=J$ and
\[
\sum_{i=0}^{j_{n+1}-1}h_{\left(\tau\left(J\right)\right)_{i}}^{\left(n\right)}<l\leq\sum_{i=0}^{j_{n+1}}h_{\left(\tau\left(J\right)\right)_{i}}^{\left(n\right)}.
\]
Thus, setting $l'=l-\sum_{i=1}^{j_{n+1}-1}h_{\left(\tau\left(J\right)\right)_{i}}^{\left(n\right)}$,
$x'=\left(T^{\left(n\right)}\right)^{j_{n+1}}\left(x\right)$ and
using the definition of $\xi_{n+1}$, we may write
\begin{equation}
\sum_{k=0}^{l}\varphi\left(T_{\alpha_{0}}^{k}\left(x\right)\right)=\sum_{k=0}^{j_{n+1}-1}\varphi_{\left(\tau_{n}\left(J\right)\right)_{k}}^{\left(n\right)}+\sum_{k=0}^{l'}\varphi\left(T_{\alpha_{0}}^{k}\left(x'\right)\right)=\xi_{n+1}\left(\left(J,j_{n+1}\right)\right)+\sum_{k=0}^{l'}\varphi\left(T_{\alpha_{0}}^{k}\left(x'\right)\right).\label{eq: Splitting of Birkhoff sum}
\end{equation}
The previous equality is obtained by splitting the Birkhoff sum up
to $l$ of a point at the base of the tower $Z_{J_{n+1}}^{\left(n+1\right)}$
into special Birkhoff sums over towers obtained at the $n^{th}$ stage
of the renormalization procedure and a remainder given by $\sum_{k=0}^{l'}\varphi\left(T_{\alpha_{0}}^{k}\left(x'\right)\right)$.
Now, by definition of the coding map $\Psi$, $T_{\alpha_{0}}^{l'}\left(x'\right)\in\Psi^{-1}\left(\left[\left(J_{1},j_{1}\right),...,\left(J_{n},j_{n}\right)\right]\right)$.
Thus, if for $y\in\bbR$, we let $A-y$ denote the set $\left\{ a-y:\ a\in A\right\} $,
(\ref{eq: Splitting of Birkhoff sum}) implies,
\[
\sum_{k=0}^{l}\varphi\left(T_{\alpha_{0}}^{k}\left(x\right)\right)\in A\iff\sum_{k=0}^{l'}\varphi\left(T_{\alpha_{0}}^{k}\left(x'\right)\right)\in A-\xi_{n+1}\left(J,j_{n+1}\right).
\]
and the equality (\ref{eq:Correspondence between sums on Markov chain and Birkhoff sums})
now follows from the hypothesis of induction, which gives 
\[
\sum_{k=0}^{l'}\varphi\left(T_{\alpha_{0}}^{k}\left(x'\right)\right)\in A-\xi_{n+1}\left(J,j_{n+1}\right)\iff\sum_{k=1}^{n}\xi_{k}\left(X_{k}\left(J_{k},j_{k}\right)\right)\in A-\xi_{n+1}\left(J,j_{n+1}\right).
\]
Since by Proposition \ref{prop: Markovian measure}, $\mu_{n}^{J}\left(\left[\omega\right]\right)=\frac{1}{h_{J}^{\left(n\right)}}$
for any $\omega\in\Sigma_{n}$, and since by the proof of Proposition
\ref{prop: Coding}, the levels of the tower $Z_{J}^{\left(n\right)}$
are in bijective correspondence with cylinders of length $n$ in $\Sigma_{n}$,
the proof is complete. 
\end{proof}

\section{\label{sec:The-CLT-for-Markov chains}the clt for markov chains}

In the previous section we established that the study of intermediate
Birkhoff sums can be reduced to the study of (in general) non-homogeneous
Markov chains. In this section we establish some (mostly well-known)
statements about such Markov chains which we use in the proof of our
temporal CLT. The main result which we need is the CLT for non-homogeneous
Markov chains. To the best of our knowledge, this was initially established
by Dobrushin \cite{dobrushin1956central1,dobrushin1956central2} (see
also \cite{sethuraman2005martingale} for a proof using martingale
approximations). Dobrushin's CLT is not directly valid in our case
(since it assumes that the contraction coefficient is strictly less
than $1$ for every transition matrix in the underlying chain, while
under our assumptions this is only valid for a product of a constant
number of matrices). While the proof of Dobrushin's theorem can be
reworked to apply to our assumptions, we do not do it here, and instead
use a general CLT for $\varphi$-mixing triangular arrays of random
variables by Utev. 

\subsection{Contraction coefficients, mixing properties and CLT for Markov chains }

In this section we collect some probability theory results for (arrays
of) non-homogeneous Markov chains that we will use in the next section. 

\smallskip{}
Let $\left(\Omega,\mathcal{B},P\right)$ be a probability space and
let $\mathcal{F}$, $\mathcal{G}$ be two sub $\sigma$-algebras of
$\mathcal{B}$. For any $\sigma$-algebra $\mathcal{A}\subset\mathcal{B}$,
denote by $\mathcal{L}^{2}\left(\mathcal{A}\right)$ the space of
square integrable, real functions on $\Omega$, which are measurable
with respect to $\mathcal{A}$. We use two measures of dependence
between $\mathcal{F}$ and $\mathcal{G}$, the so called $\varphi$-coefficient
and $\rho$-coefficient, defined by 
\[
\varphi\left(\mathcal{F},\mathcal{G}\right):=\sup_{A\in\mathcal{G},B\in\mathcal{F}}\left|P\left(A\left|B\right.\right)-P\left(A\right)\right|
\]
and 
\[
\rho\left(\mathcal{F},\mathcal{G}\right):=\sup_{f\in\mathcal{L}^{2}\left(\mathcal{F}\right),g\in\mathcal{L}^{2}\left(\mathcal{G}\right)}\left|\frac{Cov\left(f,g\right)}{\sqrt{Var\left(f\right)Var\left(g\right)}}\right|.
\]
It is a well-known fact (see \cite{bradley2005basic}) that 
\begin{equation}
\rho\left(\mathcal{F},\mathcal{G}\right)\leq2\left(\varphi\left(\mathcal{F},\mathcal{G}\right)\right)^{\frac{1}{2}}.\label{eq: Correspondence between mixing coefficients}
\end{equation}
In what follows, let $Y=\left\{ Y_{1}^{\left(n\right)},...,Y_{n}^{\left(n\right)}:\,n\geq1\right\} $
be a triangular array of mean zero, square integrable random variables
such that the random variables in each row are defined on the same
probability space $\left(\Omega,\mathcal{B},P\right)$. For any set
$\mathcal{Y}$ of random variables defined on $\left(\Omega,\mathcal{B},P\right)$,
let us denote by $\sigma\left(\mathcal{Y}\right)$ to be the $\sigma$-algebra
generated by all the random variables in $\mathcal{Y}.$ 

Set $S_{n}=\sum_{k=1}^{n}Y_{k}^{\left(n\right)}$ and $e_{n}=E\left(S_{n}\right)$,
$\sigma_{n}=\sqrt{Var\left(S_{n}\right)}$. For any $n,k\in\mathbb{N}$
let 
\[
\varphi_{n}\left(k\right):=\sup\limits _{1\leq s,s+k\leq n}\varphi\left(\sigma\left(Y_{i}^{\left(n\right)},i\leq s\right),\sigma\left(Y_{i}^{\left(n\right)},i\geq s+k\right)\right),
\]
\[
\varphi\left(k\right):=\sup_{n}\varphi_{n}\left(k\right).
\]
The array $Y$ is said to be $\varphi$-\emph{mixing} if $\varphi\left(k\right)\rightarrow0$
as $k$ tends to infinity. 

\medskip{}

The following CLT for $\varphi$-mixing arrays of random variables,
which follows from a more general CLT for such arrays in \cite{utev1991central}
is the main result that we use to prove our distributional CLT. 
\begin{thm}
\label{thm: CLT for phi-mixing}Let $Y$ be a $\varphi$-mixing array
of square integrable random variables and assume that 
\begin{equation}
\lim_{n\to\infty}\sigma_{n}^{-2}\sum_{k=1}^{n}E\left(Y_{k}^{\left(n\right)}\ind_{\left\{ \left|Y_{k}^{\left(n\right)}\right|>\epsilon\sigma_{n}\right\} }\right)=0\label{eq:LFcondition}
\end{equation}
for every $\epsilon>0$. Then 
\[
\frac{S_{n}-e_{n}}{\sigma_{n}}
\]
converges in law to the standard normal distribution. 
\end{thm}
Let $\mathcal{T},\mathcal{S}$ be finite sets and $P$ a stochastic matrix with entries
indexed by $\mathcal{T}\times \mathcal{S}$. The contraction coefficient of $P$ is defined
by 
\begin{equation}
\tau\left(P\right)=\frac{1}{2}\sup\limits _{x_{1},x_{2}\in \mathcal{T}}\sum_{s\in \mathcal{S}}\left|P_{x_{1},s}-P_{x_{2},s}\right|.\label{eq: Contraction coefficient}
\end{equation}

It is not difficult to see that $\tau\left(P\right)=0$ if and only
if the entry $P_{s,t}$ does not depend on $s$ and that 
\begin{equation}
\tau\left(PQ\right)\leq\tau\left(P\right)\tau\left(Q\right)\label{eq: Submult of Erg. coefficient}
\end{equation}
 for any pair of stochastic matrices $P$ and $Q$ such that their
product is defined. 

For $n\in\bbN$, let $X_{1}^{\left(n\right)},...,X_{n}^{\left(n\right)}$
be a Markov chain with each $X_{i}^{\left(n\right)}$ taking values
in a finite state space $\mathcal{S}_{i}$, determined by an initial
distribution $\pi_{n}$ and transition matrices $P_{i}^{\left(n\right)}$,
$i=1,...n$ (thus, each matrix $P_{i}^{\left(n\right)}$ has dimension
$\left|\mathcal{S}_{i}\right|\times\left|\mathcal{S}_{i+1}\right|$). 
\begin{prop}
\label{prop: Contraction coefficient implies phi mixing}Assume that
there exist $0\leq\delta<1$ and $s\in\bbN$ such that for every $n\in\bbN$
\[
\tau\left(P_{k}^{\left(n\right)}...P_{k+s}^{\left(n\right)}\right)<\delta,\qquad\text{for\,any\ }1\leq k\leq n-s.
\]
Then $X=X_{0}^{\left(n\right)},...,X_{n}^{\left(n\right)}$ is $\varphi$-mixing
and $\varphi\left(k\right)$ tends to $0$ as $k\rightarrow\infty$
with exponential rate. 
\end{prop}
\begin{proof}
This is a direct consequence of the inequality 
\[
\varphi\left(\sigma\left(Y_{i}^{\left(n\right)},i\leq j\right),\sigma\left(Y_{i}^{\left(n\right)},i\geq j+k\right)\right)\leq\tau\left(P_{j}^{\left(n\right)}...P_{j+k}^{\left(n\right)}\right)\qquad\text{for}\ 1\leq j\leq n-k
\]
(see relation (1.1.2) and Proposition 1.2.5 in \cite{iosifescu1969random})
and the fact that $\tau\left(P_{j}^{\left(n\right)}...P_{j+k}^{\left(n\right)}\right)\leq\delta^{\left[\frac{k}{s}\right]}$,
which immediately follows from the assumption and (\ref{eq: Submult of Erg. coefficient}). 
\end{proof}
Now, let $\xi_{i}^{\left(n\right)}:\mathcal{S}_{i}\rightarrow\bbR$
, with $1\leq i\leq n$ for any $n\in\mathbb{N}$, be an array of
functions and set $Y_{i}^{\left(n\right)}=\xi_{i}^{\left(n\right)}\left(X_{i}^{\left(n\right)}\right)$.
Henceforth, we assume that 
\begin{equation}
\sup\left\{ \left|\xi_{i}^{\left(n\right)}\left(s\right)\right|:\,n\in\bbN,\,i=0,...,n,\ s\in\mathcal{S}_{i}\right\} =M<\infty.\label{eq: Uniform boundness of function array}
\end{equation}
An application of Theorem \ref{thm: CLT for phi-mixing} yields the
following corollary. 
\begin{cor}
\label{cor:Under-the-conditions}Under the conditions of Proposition
\ref{prop: Contraction coefficient implies phi mixing}, assume further
that and $\sigma_{n}\rightarrow\infty$. Then $\frac{S_{n}-e_{n}}{\sigma_{n}}$
converges in law to the standard normal distribution.
\end{cor}
\begin{proof}
It is enough to remark that the condition (\ref{eq:LFcondition})
in Theorem \ref{thm: CLT for phi-mixing} holds trivially for $n$
large in virtue of the bound in (\ref{eq: Uniform boundness of function array})
since by assumption the variance $\sigma_{n}\rightarrow\infty.$ 
\end{proof}
Let now $\tilde{\pi}_{n}$ be a sequence of probability distributions
on $\mathcal{S}_{1}$, and let $\tilde{X}_{1}^{\left(n\right)},...,\tilde{X}_{n}^{\left(n\right)}$
be an array of Markov chains generated by initial distributions $\tilde{\pi}_{n}$
and transition matrices $P_{i}^{\left(n\right)}$. Let $\tilde{S}_{n}=\sum_{k=0}^{n-1}\xi_{i}\left(\tilde{X}_{i}\right)$
and let $\tilde{e}_{n}=E\left(\tilde{S}_{n}\right)$, $\tilde{\sigma}_{n}=\sqrt{Var\left(\tilde{S}_{n}\right)}$. 
\begin{prop}
\label{prop: Same order of variances and expectations}Under the conditions
of Proposition \ref{prop: Contraction coefficient implies phi mixing},
there exists a constant $C$, independent of the sequences $\pi_{n}$
and $\tilde{\pi}_{n}$, such that $\left|e_{n}-\tilde{e}_{n}\right|\leq C$
and $\left|\sigma_{n}^{2}-\tilde{\sigma}_{n}^{2}\right|\leq C$ for
all $n\in\bbN$. 
\end{prop}
\begin{proof}
The assumption implies that there exists a constant $M$ and a sequence
of rank $1$ stochastic matrices (i.e stochastic matrices with all
rows being identical) $V_{i}^{\left(n\right)}$ such that 
\[
\sup\left\{ \left\Vert V_{i}^{\left(n\right)}-\prod_{j=1}^{i}P_{j}^{\left(n\right)}\right\Vert ,\ i=1,...,n,\ n\in\bbN\right\} \leq M\delta^{\frac{i}{s}}
\]
 (see \cite[Chapter 4, Cor. 2]{seneta2006non}), where for two matrices
$P$, $Q$ indexed by $S\times T$, $\left\Vert P-Q\right\Vert =\max\left\{ \left|P_{s,t}-Q_{s,t}\right|:\,\left(s,t\right)\in S\times T\right\} $.
Using (\ref{eq: Uniform boundness of function array}) it follows
that there exists a constant $\tilde{C}$ which depends only on the
array of matrices $P_{i}^{\left(n\right)}$ and functions $\xi_{i}^{\left(n\right)}$,
such that 
\[
\left|E\left(\xi_{i}^{\left(n\right)}\left(X_{i}\right)\right)-E\left(\xi_{i}^{\left(n\right)}\left(\tilde{X}_{i}\right)\right)\right|=\left|\sum_{s\in S_{1}}\sum_{t\in S_{i}}\left(\pi_{n}\left(s\right)-\tilde{\pi}_{n}\left(s\right)\right)\left(P_{1}^{\left(n\right)}...P_{i-1}^{\left(n\right)}\right)_{s,t}\cdot\xi_{i}^{\left(n\right)}\left(t\right)\right|\leq\tilde{C}\delta^{\frac{n}{s}}.
\]
 Since the right hand side of the last inequality is a general term
of a summable geometric series, we have proved that there exists a
constant $C$, such that $\left|e_{n}-\tilde{e}_{n}\right|\leq C$
for all $n\in\bbN$. 

To prove the inequality for the variances, we first note that it follows
from (\ref{eq: Correspondence between mixing coefficients}) and (\ref{eq: Uniform boundness of function array})
that there exists a constant $C'$ independent of $\pi_{n}$, such
that 
\begin{equation}
\left|\sum_{1\leq i<j\leq n}Cov\left(\xi_{i}^{\left(n\right)}\left(X_{i}^{\left(n\right)}\right)\cdot\xi_{j}^{\left(n\right)}\left(X_{j}^{\left(n\right)}\right)\right)\right|<C'\label{eq: Covariance bound}
\end{equation}
 for all $n\in\bbN$. An analogous inequality hence holds also for
the array $\tilde{X}_{i}^{\left(n\right)}$ instead of $\tilde{X}_{i}^{\left(n\right)}$,
so that 
\begin{equation}
\left|\sum_{1\leq i<j\leq n}Cov\left(X_{i}^{(n)},X_{j}^{(n)}\right)-Cov\left(\tilde{X}_{i}^{(n)},\tilde{X}_{j}^{(n)}\right)\right|<2C'.\label{eq:Cov double boud}
\end{equation}
Moreover, since $\sup_{n}\left|\mu_{n}-\tilde{\mu}_{n}\right|<\infty$,
one can also prove that 
\begin{equation}
\sup_{n}\left|\sum_{i=1}^{n}Var\left(\xi_{i}^{\left(n\right)}\left(X_{i}^{\left(n\right)}\right)\right)-Var\left(\xi_{i}^{\left(n\right)}\left(\tilde{X}_{i}^{\left(n\right)}\right)\right)\right|<\infty.\label{eq:control variances}
\end{equation}
Now, write 
\begin{align*}
 &  & \left|\sigma_{n}^{2}-\tilde{\sigma}_{n}^{2}\right| & \leq\left|\sum_{i=1}^{n}\left(Var\left(\xi_{i}^{\left(n\right)}\left(X_{i}^{\left(n\right)}\right)\right)-Var\left(\xi_{i}^{\left(n\right)}\left(\tilde{X}_{i}^{\left(n\right)}\right)\right)\right)\right|\\
 &  &  & +\left|\sum_{1\leq i<j\leq n}Cov\left(\xi_{i}^{\left(n\right)}\left(X_{i}^{\left(n\right)}\right),\xi_{j}^{\left(n\right)}\left(X_{j}^{\left(n\right)}\right)\right)-Cov\left(\xi_{i}^{\left(n\right)}\left(\tilde{X}_{i}^{\left(n\right)}\right),\xi_{j}^{\left(n\right)}\left(\tilde{X}_{j}^{\left(n\right)}\right)\right)\right|.
\end{align*}
 The proof of the Lemma hence follows by (\ref{eq:Cov double boud})
and (\ref{eq:control variances}).
\end{proof}

\section{proof of the temporal clt \label{sec:Proof-of-the TCLT}}

In this section we give the proof of Theorem \ref{thm: Main thm}.
We need to show that we can apply the results on Markov chains summarized
in the previous section (and in particular Corollary \ref{cor:Under-the-conditions})
to the Markov chains that model the dynamics. In order to check that
the required assumptions are verified, we first show, in section \ref{subsec:Positivity-of-products-1}
a result on positivity of the product of finitely many transition
matrices, which follows from the assumption that $\alpha$ is badly
approximable and $\beta$ is badly approximable with respect to $\alpha$.
Then, in section \ref{subsec:Growth-of-the} we prove that the variance
grows. Finally, the proof of the Theorem is given in section \ref{subsec:Proof-of-Theorem}.

\subsection{\label{subsec:Positivity-of-products}Positivity of products of incidence
matrices\label{subsec:Positivity-of-products-1}}

Let us recall that in Section \ref{subsec:Continued-fraction-renormalization}
we described a renormalization procedure that, to a pair of parameters
$\left(\alpha,\beta\right)$ (under the assumption that $\left(\alpha,\beta\right)\in\tilde{X}$),
in particular associates a sequence $\left(A_{n}\right)_{n}$ of matrices
(given by equations (\ref{eq: Transition mat 1}), (\ref{eq: Transition mat 2})
and (\ref{eq: Transition mat 3}) respectively), which are the incidence
matrices of the sequence of substitutions $\left(\tau_{n}\right)_{n}$
which describe the tower structure. In this section, we develop conditions
on the pair $\left(\alpha,\beta\right)\in\tilde{X}$ that ensure that
we may split the sequence of incidence matrices $\left(A_{n}\right)_{n}$
associated to $\left(\alpha,\beta\right)$ into consecutive blocks
of uniformly bounded length, so that the product of matrices in each
block is strictly positive. This fact is used for showing that the
Markov chain associated to $\left(\alpha,\beta\right)$ satisfies
the assumption of the previous section needed to prove the CLT. 

Under the assumption that $\left(\alpha,\beta\right)\in\tilde{X}$,
the orbit $\hat{G}^{n}\left(\alpha,\beta\right)$ of the point $\left(\alpha,\beta\right)$
under the transformation $\hat{G}$ defined in (\ref{eq:Ghat def})
is infinite and one can consider its\emph{ itinerary} with respect
to the partition $\left\{ X_{G,}X_{B_{-}},X_{B_{+}}\right\} $ defined
in Section \ref{subsec:Continued-fraction-renormalization}: the itinerary
is the sequence $\left(s_{n}\right)_{n}\in\mathcal{S}^{\mathbb{N}\cup\text{\ensuremath{\left\{  0\right\} } }},$
where $\mathcal{S}:=\left\{ G,B_{-},B_{+}\right\} ,$ defined by

\begin{equation}
s=J\in\left\{ G,B_{-},B_{+}\right\} \iff\hat{G}^{n}\left(\alpha,\beta\right)\in X_{J},\quad n=0,1,2,....\label{eq: Itinerary wrt (G,B)}
\end{equation}
We will call $\mathcal{S}:=\left\{ G,B_{-},B_{+}\right\} $ the set
of \emph{states} and we will say that $s\left(\alpha,\beta\right):=\left(s_{n}\right)_{n}\in\mathcal{S}^{\mathbb{N}\cup\text{\ensuremath{\left\{  0\right\} } }}$
the infinite sequence of \emph{states} associated to $\left(\alpha,\beta\right)\in\tilde{X}$.
From the definitions in Section \ref{subsec:Continued-fraction-renormalization},
$s_{n}=G$ (or $B_{-},B_{+}$ respectively) if and only if the incidence
matrix $A_{n}$ is of the form (\ref{eq: Transition mat 1}) (or (\ref{eq: Transition mat 2}),
(\ref{eq: Transition mat 3}) respectively). It can be easily deduced
from the description of the renormalization procedure that not all
sequences in $\mathcal{S}^{\bbN}$ are images of some pair $\left(\alpha,\beta\right)\in\tilde{X}$.
The sequences $s\in\mathcal{S}^{\mathbb{N}\cup\text{\ensuremath{\left\{  0\right\} } }}$
such that $s=s\left(\alpha,\beta\right),$ for some $\left(\alpha,\beta\right)\in\tilde{X}$)
form a stationary Markov compactum $\tilde{\mathcal{S}}$$\subset\mathcal{S}^{\mathbb{N}\cup\text{\ensuremath{\left\{  0\right\} } }}$
with state space determined by the graph, 
\[
\xymatrix{G\ar@/^{.7pc}/[r]\ar@(ul,dl) & B_{+}\ar@/^{.7pc}/[l]\ar@/^{.7pc}/[r] & B_{-}\ar@/^{.7pc}/[l]}
\]
namely $s=\left(s_{n}\right)_{n}\in\tilde{\mathcal{S}}$ if and only
if for any $n\geq0$ there is an oriented edge from the state $s_{n}\in\mathcal{S}$
to the state $s_{n+1}\in\mathcal{S}$ in the graph above. 

Since at this point we are interested solely in positivity of the
incidence matrices and not in the values themselves, we define a function
$F:\mathcal{S}\rightarrow M_{3}\left(\bbZ\right)$, where $M_{3}\left(\bbZ\right)$
are $3\times3$ matrices, by 

\begin{align}
F\left(s\right) & =\begin{bmatrix}0 & 1 & 1\\
1 & 0 & 1\\
0 & 1 & 0
\end{bmatrix}\qquad\text{if \quad s=G;}\label{eq: The function F(s)}\\
F\left(s\right) & =\begin{bmatrix}1 & 0 & 1\\
0 & 1 & 0\\
1 & 0 & 0
\end{bmatrix}\qquad\text{if \quad\ensuremath{s}=\ensuremath{B_{-}};}\nonumber \\
F\left(s\right) & =\begin{bmatrix}1 & 0 & 1\\
1 & 1 & 0\\
1 & 0 & 0
\end{bmatrix}\qquad\text{if \quad\ensuremath{s}=\ensuremath{B_{+}}.}\nonumber 
\end{align}

Note that $F$ is defined in such a way, so that some entry of the
matrix $F\left(s_{n}\right)$ is $1$, if and only if the corresponding
entry of incidence matrix $A_{n}$ which corresponds to the state
$s\left(\alpha,\beta\right):=\left(s_{n}\right)_{n}$ has a non-zero
value, independently of $a_{n}$ and $b_{n}$ (for example $a_{n}$
and $a_{n}-b_{n}+1$ are always greater than $1$ or \textbf{$b_{n}\geq1$}
when $\left(\alpha,\beta\right)\in G$). Note that the other implication
is not necessarily true, namely some entries of $F(s)$ could be $0$
even if the corresponding entry of the incidence matrices are positive
(in such cases the positivity depends on the values of $a_{n}$ and
$b_{n}$, for example $a_{n}-b_{n}$ is zero if $a_{n}=b_{n}$). Thus,
for any $n,k\in\bbN\bigcup\left\{ 0\right\} $, 

\[
F\left(s_{n+k}\right)F\left(s_{n+k-1}\right)...F\left(s_{n}\right)>0\implies A_{n+k}A_{n+k-1}...A_{n}>0.
\]
It immediately follows from the topology of the transition graph that
every itinerary $s\in\mathcal{S}^{\bbN\cup\left\{ 0\right\} }$ can
be written in the form 
\begin{equation}
s=W_{1}\left(B_{-}B_{+}\right)^{n_{1}}W_{2}\left(B_{-}B_{+}\right)^{n_{2}}...W_{k}\left(B_{-}B_{+}\right)^{n_{k}}...\label{eq: Itinerary decompoistion 1}
\end{equation}
where $W_{k}$, $k\in\bbN$ are words in the alphabet $\mathcal{S}$
which do not contain $B_{-}$ (i.e. they are words in $G$ and $B_{+})$,
and $W_{k}$ is not empty for $k\geq2$. Note that it may be that
the number of appearances of $B_{-}$ in the above representation
is finite. This means that there exists $K$ such that $n_{k}=0$
for $k\geq K$ and in this case the above representation reduces to
\begin{equation}
s=W_{1}\left(B_{-}B_{+}\right)^{n_{1}}W_{2}\left(B_{-}B_{+}\right)^{n_{2}}...W_{K}\left(B_{-}B_{+}\right)^{n_{K}}W_{K+1}\label{eq: Itinerary decomposition 2}
\end{equation}
 where the length of $W_{K+1}$ is infinite. 
\begin{defn}
Let $\left(\alpha,\beta\right)\in\tilde{X}$. We say that $\beta$\textit{
is of Ostrowski bounded type with respect to} $\alpha$ if the decomposition
of $s\left(\alpha,\beta\right)\in\mathcal{S^{\mathbb{N}}}$ given
by (\ref{eq: Itinerary decompoistion 1}) or (\ref{eq: Itinerary decomposition 2})
satisfies $\sup\left\{ n_{k}\right\} =M<\infty$, where the supremum
is taken over $k\in\bbN$ in the first case, and over $k\in\left\{ 1,...,K\right\} $
in the second case. We say in both cases that $\beta$\textit{ is
of Ostrowski bounded type of order $M$.}
\end{defn}
\begin{prop}
\label{prop: Grouping Prop}Let $\beta$ be of Ostrowski bounded type
of order $M$ with respect to $\alpha$ and let $(A_{i})_{i}$ bee
i the sequence of incidence matrices associated to $(\alpha,\beta)$
by the Ostrowski renormalization. Then for any $k$, and any $n\geq5M$,
we have that $A_{k+n}A_{k+n-1}...A_{k}>0$.
\end{prop}
\begin{proof}
Let $W_{1}\left(B_{-}B_{+}\right)^{n_{1}}W_{2}\left(B_{-}B_{+}\right)^{n_{2}}...W_{k}\left(B_{0}B_{1}\right)^{n_{k}}...$
the decomposition of $s\left(\alpha,\beta\right)$ described above.
Direct calculation gives that the product of matrices which corresponds
to an admissible word of length $5$ (or more) which does not contain
$B_{-}$ is strictly positive. Also, any word of length $5$ which
starts with $B_{-}B_{+}G$ gives a transition matrix which is strictly
positive. Note that it follows from the transition graph that each
$W_{i}$, $i\geq2$ must start with $G$ and must be of length strictly
greater than $1$. Since any subword of length greater than $5M$
must contain a block of the form $B_{-}B_{+}W_{i}B_{-}$, or a block
of length at least $5$ where there is no occurrence of $B_{-}$,
the claim follows.
\end{proof}
\begin{lem}
\label{lem: badly approximable implies Ostrowski bounded}If $0<\alpha<\frac{1}{2}$
is badly approximable and $\beta\in\left(0,1\right)$ is badly approximable
with respect to $\alpha$, then the pair $\left(\alpha_{0},\beta_{0}\right),$
related to $\left(\alpha,\beta\right)$ via equations (\ref{eq:Correcpondence between alpha and alpha_zero})
and (\ref{eq: Correcspondence between beta and beta_zero}), satisfies
$\left(\alpha_{0},\beta_{0}\right)\in\tilde{X}$ and $\beta_{0}$
is of Ostrowski bounded type with respect to $\alpha_{0}$.
\end{lem}
\begin{proof}
Let $\sum_{k=0}^{\infty}x^{\left(k\right)}$ be the Ostrowski expansion
of $\beta_{0}$ in terms of $\alpha_{0}$ given by Proposition \ref{prop: Ostrowsky expansion}.
Then by Remark \ref{rem: Ostrowski Remark} $\sum_{k=0}^{n}x^{\left(k\right)}\in\left\{ T_{\alpha_{0}}^{j}\left(0\right):\ 0\leq j\leq q_{n-1}+q_{n}\right\} \bigcup\left\{ \alpha_{0}\right\} $
where $q_{n}$ are the denominators of the $n^{th}$ convergent in
the continued fraction expansion of $\alpha$. Since under the conjugacy
$\psi$ between $T_{\alpha_{0}}$ and $R_{\alpha}$ (where both maps
are viewed as rotations on a circle), the (equivalence class of the)
points $0$ and $\alpha_{0}$ in the domain of $T_{\alpha_{0}}$ correspond
respectively to the (equivalence class of) points $1-\alpha$ and
$1$ in the domain of $R_{\alpha}$, we obtain that $\psi^{-1}\left(\sum_{k=0}^{n}x^{\left(k\right)}\right)\in\left\{ R_{\alpha}^{j}\left(1-\alpha\right):\ 0\leq j\leq q_{n-1}+q_{n}-1\right\} $.
It follows that the Ostrowski expansion of $\beta_{0}$ is infinite,
since otherwise, if there exists an $n$ such that $\beta_{0}=\sum_{k=0}^{n}x^{\left(k\right)}$,
we would get that $\beta=\psi^{-1}\left(\beta_{0}\right)=1-\alpha+j\alpha\mod1$
for some $j\in\bbN\bigcup\left\{ 0\right\} $, which obviously contradicts
(\ref{eq:BAbeta}). Thus, $\left(\alpha_{0},\beta_{0}\right)\in\tilde{X}$. 

Fix $M\in\bbN$ and let $s=s\left(\alpha_{0},\beta_{0}\right)$ be
defined by (\ref{eq: Itinerary wrt (G,B)}). We claim that, if for
some $n\in\bbN$, $s_{n+i}\in\left\{ B_{-},B_{+}\right\} $ for all
$1\leq i\leq M$, then there exist a constant $C$, which does not
depend on $n$, and $0\leq k\leq q_{n}+q_{n-1}$, $p\in\bbZ$, such
that 
\begin{equation}
\left|\beta-\left(k-1\right)\alpha-p\right|\leq\frac{C}{q_{n+M}}.\label{eq:contradiction bounded M}
\end{equation}
 The second assertion of the Lemma follows immediately from this and
the fact that $q_{n+M}\rightarrow\infty$ as $M$ tends to $\infty$. 

To see that the claim holds, suppose that $s_{n+i}\in\left\{ B_{-},B_{+}\right\} $
for all $0\leq i\leq M$. Recalling the description of the renormalization
procedure in section \ref{subsec:Continued-fraction-renormalization},
this is equivalent to $x^{\left(n+i\right)}=0$ for all $0\leq i\leq M$,
so that$\sum x^{\left(k\right)}=\sum_{k=0}^{n+M}x^{\left(k\right)}$.
Thus, by the estimate of the reminder in an Ostrowski expansion given
by Proposition \ref{prop: Ostrowsky expansion}, we obtain that
\[
\left|\beta_{0}-\sum_{k=0}^{n}x^{\left(k\right)}\right|=\left|\beta_{0}-\sum_{k=0}^{n+M}x^{\left(k\right)}\right|\leq\left|\beta^{\left(n+M+1\right)}\right|\leq\alpha^{\left(n+M\right)}.
\]
Since $\alpha$ is badly approximable, $\alpha^{\left(n\right)}=\mathcal{G}^{n}\left(\alpha\right)\leq\frac{C}{q_{n}}$
for all $n$, where $C$ is a constant which depends only on $\alpha$.
Since the conjugacy map $\psi$ is affine, the previous inequality
yields that there exists a constant $C$, such that 
\[
\left|\beta-\psi^{-1}\left(\sum_{k=0}^{n}x^{\left(k\right)}\right)\right|\leq\frac{C}{q_{n+M}}.
\]
Since $\psi^{-1}\left(\sum_{k=0}^{n}x^{\left(k\right)}\right)\in\left\{ R_{\alpha}^{j}\left(1-\alpha\right):\ 0\leq j<q_{n-1}+q_{n}\right\} $,
we obtain that 
\[
\psi^{-1}\left(\sum_{k=0}^{n}x^{\left(k\right)}\right)=1-\alpha+k\alpha+p
\]
where $0\leq k<q_{n}+q_{n-1}$, and $p\in\bbZ$. Thus, combining the
last two equations, we proved (\ref{eq:contradiction bounded M}).
This completes the proof of the Lemma. 
\end{proof}
Let $0<\alpha<\frac{1}{2}$ be badly approximable, let $\beta\in\left(0,1\right)$
be badly approximable with respect to $\alpha$ and let $\left(\alpha_{0},\beta_{0}\right)$
be related to $\left(\alpha,\beta\right)$ via equations (\ref{eq:Correcpondence between alpha and alpha_zero})
and (\ref{eq: Correcspondence between beta and beta_zero}). Since
by the previous proposition $\left(\alpha_{0},\beta_{0}\right)\in\tilde{X}$,
the sequence of transition matrices $p^{\left(n\right)}$ associated
to the pair $\left(\alpha_{0},\beta_{0}\right)$ given by Definition
\ref{def: Markov transitions} is well defined. Recall that $\tau\left(P\right)$, where $P$ is a stochastic matrix, denotes the contraction coefficient
defined by (\ref{eq: Contraction coefficient}). 
\begin{cor}
\label{cor: Positivity of contraction coefficient for bbd type}Let
$0<\alpha<\frac{1}{2}$ be badly approximable, $\beta\in\left(0,1\right)$
be badly approximable with respect to $\alpha$ and let $\left(\alpha_{0},\beta_{0}\right)$
be related to $\left(\alpha,\beta\right)$ via equations (\ref{eq:Correcpondence between alpha and alpha_zero})
and (\ref{eq: Correcspondence between beta and beta_zero}). Then
if $p^{\left(n\right)}$ is the sequence of transition matrices associated
to $\left(\alpha_{0},\beta_{0}\right)$ (see Definition \ref{def: Markov transitions}),
there exist $M\in\bbN$, and $0\leq\delta<1$, such that 
\[
\tau\left(p^{\left(n+M\right)}\cdot p^{\left(n+M-1\right)}\cdot...\cdot p^{\left(n\right)}\right)<\delta\qquad\text{for\, all}\ n\in\bbN.
\]
 
\end{cor}
\begin{proof}
Lemma \ref{lem: badly approximable implies Ostrowski bounded} implies
that $\beta_{0}$ is of Ostrowski bounded type. By definition of the
transition matrices $p^{\left(n\right)}$ (see Definition \ref{def: Markov transitions}),
for any $\left(K,k\right)\in \mathcal{S}_{n+M+1}$, $\left(J,j\right)\in \mathcal{S}_{n}$
\[
\left(p^{\left(n+M\right)}\cdot p^{\left(n+M-1\right)}\cdot...\cdot p^{\left(n\right)}\right)_{\left(K,k\right),\left(J,j\right)}>0
\]
if and only if 
\[
\left(A_{n+M}A_{n+M-1}...A_{n}\right)_{\left(\tau_{n+M}\left(K\right)\right)_{k},J}>0.
\]
This should be interpreted as the statement that the probability to
pass from a state $\left(K,k\right)\in\mathcal{S}_{M+n+1}$ to some
state $\left(J,j\right)\in\mathcal{S}_{n}$ is positive if and only
if the intersection of the tower $Z_{J}^{\left(n\right)}$ with the
subtower of $Z_{K}^{\left(n+M+1\right)}$ labelled by $\left(K,k\right)$
is non-empty. Thus, Proposition \ref{prop: Grouping Prop} implies
that there exists $M\in\bbN$ such that $p^{\left(n+M\right)}\cdot...\cdot p^{\left(n\right)}$
is strictly positive for any $n\in\bbN$. From $\alpha$ being badly
approximable (see inequality (\ref{eq: Inequality for heights of towers}))
and by the fact that by definition, every positive entry of $p^{\left(n+M\right)}\cdot...\cdot p^{\left(n\right)}$
is a ratio between the heights of tower at the $\left(n+M\right)^{th}$
and $n^{th}$ stage of the renormalization, it follows that there
exists $\delta>0$ which is independent of $n$, such that every entry
of $p^{\left(n+M\right)}\cdot...\cdot p^{\left(n\right)}$ is not
less than $\delta$. Note that it follows from the definition of the
coefficient $\tau$ (see (\ref{eq: Contraction coefficient})) that
if $P_{n\times m}$ is a stochastic matrix such that there exists
$\delta>0$, for which $P_{i,j}>\delta$, for all $1\leq i\leq n$,
$1\leq j\leq m$, then $\tau\left(P\right)<1-\delta$. Thus, the proof
is complete. 
\end{proof}

\subsection{Growth of the variance\label{subsec:Growth-of-the}}

In this section we consider the random variables $\xi_{k}\left(X_{k}\right),\,k\in\mathbb{N}$,
constructed in Section \ref{subsec: Markov chain-1} (see equation
(\ref{eq: Constructed Random variables}) therein). Recall that the
array is well defined for any given pair of parameters $\left(\alpha_{0},\beta_{0}\right)\in\tilde{X}$
and, by the key Proposition \ref{prop: Connection between temporal sums and MC},
models Birkhoff sums over the transformation $T_{\alpha_{0}}$ of
the function $\varphi$ defined by (\ref{eq: Cocycle function}),
which has a jump at \textbf{$\beta_{0}$.} The goal in the present
section is to show that if $\varphi$ is not a coboundary, then the
variance $Var_{\mu_{n}}\left(\sum_{k=1}^{n}\xi_{k}\left(X_{k}\right)\right)$
tends to infinity as $n$ tends to infinity, where $Var_{\mu_{n}}\left(\sum_{k=1}^{n}\xi_{k}\left(X_{k}\right)\right)$
is the variance of $\sum_{k=1}^{n}\xi_{k}\left(X_{k}\right)$ with
respect to the measure $\mu_{n}$. 

Let us first recall the definition of tightness and a criterion which
characterizes coboundaries. 
\begin{defn}
Let $\left(\Omega,\mathcal{B},P\right)$ be a probability space. A
sequence of random variables $\left\{ Y_{n}\right\} $ defined on
$\Omega$ and taking values in a Polish space $\mathcal{P}$ is tight
if for every $\epsilon>0$, there exists a compact set $C\subseteq\mathcal{P}$
such that $\forall n\in\bbN$, $P\left(Y_{n}\in C\right)>1-\epsilon$. 
\end{defn}
Let $\left(X,\mathcal{B},m,T\right)$ be a probability preserving
system and let $f:X\rightarrow\bbR$ be a measurable function. We
say that $f$ is a \emph{coboundary} if there exists a measurable
function $g:X\rightarrow\bbR$ such that the equality $f\left(x\right)=g\left(x\right)-g\circ T\left(x\right)$
holds almost surely. Let us recall the following characterization
of coboundaries on $\bbR$ (see \cite{aaronson2000remarks}). 
\begin{thm}
The sequence $\left\{ \sum_{k=0}^{n-1}f\circ T^{k}\right\} $ is tight
if and only if $f$ is a coboundary\label{thm: coboundaries tight}.
\end{thm}
Set $e_{n}=E_{\mu_{n}}\left(\sum_{k=1}^{n}X_{k}\right)$, $\sigma_{n}=\sqrt{Var_{\mu_{n}}\left(\sum_{k=1}^{n}X_{k}\right)}$.
We will now prove the following lemma. 
\begin{lem}
Assume that there exists a strictly increasing \label{lem:variance tight}sequence
of positive integers $\left\{ n_{j}\right\} _{j=1}^{\infty}$ such
that 
\[
\sup\left\{ \sigma_{n_{j}}:\ j=1,2,...\right\} <\infty.
\]
Then the sequence $\varphi_{n}=\sum_{k=0}^{n-1}\varphi\circ T_{\alpha_{0}}^{k}$
is tight.
\end{lem}
Thus, combining Theorem \ref{thm: coboundaries tight} and Lemma \ref{lem:variance tight}
we have the following.
\begin{cor}
\label{cor: var goest to infty}If $\sigma_{n}$ does not tend to
infinity as $n\rightarrow\infty$, then $\varphi$ must be a coboundary. 
\end{cor}
\begin{proof}[Proof of Lemma \ref{lem:variance tight}.]
 Fix $\epsilon>0$. By Markov's inequality the assumption that $\sup\left\{ \sigma_{n_{j}}:\ j=1,2,...\right\} <\infty$,
implies that there exists a constant $A$ such that for every $j\in\bbN$,
\begin{equation}
\mu_{n_{j}}\left(\left|\sum_{k=1}^{n_{j}}\xi_{k}\left(X_{k}\right)-e_{n_{j}}\right|>A\right)<\epsilon.\label{eq: Markov inequality-1}
\end{equation}
 Let $n\in N$ and fix $j$ such that $n<\epsilon h_{J}^{\left(n_{j}\right)}$
for any $J\in\left\{ L,M,S\right\} $ (this is possible since the
heights of the towers $h_{J}^{\left(n\right)}$ tend to infinity with
$n$). Let $x$ be any point on level $l$ of the tower $Z_{J}^{\left(n_{j}\right)}$
and consider the Birkhoff sums $\varphi_{n}(x)$. Then there exists
a point $x_{0}=x_{0}(x)$ in the base of the tower $I_{J}^{\left(n_{j}\right)}$
such that $\varphi_{n}\left(x\right)=\varphi_{n+l}\left(x_{0}\right)-\varphi_{l}\left(x_{0}\right)$.
Since the values of $S_{l}(x_{0})$ for $x_{0}\in I_{J}^{\left(n_{j}\right)}$
and $0\leq l\leq h_{J}^{\left(n_{j}\right)}$ do not depend on $x_{0},$
we can choose any point $x_{J}$ $\in I_{J}^{\left(n_{j}\right)}$
and by triangle inequality we have that $\left|\varphi_{n}\left(x\right)\right|>2A$
implies that $\left|\varphi_{n+l}\left(x_{J}\right)-e_{n_{j}}\right|>A$
or $\left|\varphi_{l}\left(x_{J}\right)-e_{n_{j}}\right|>A$ for any
point $x$ on level $l$ of the tower $Z_{J}^{\left(n_{j}\right)}$
with $0\leq l<h_{J}^{\left(n_{j}\right)}-n$. Thus, 
\begin{align*}
\lambda\biggl(\left|\varphi_{n}(x)\right|> & 2A\biggl|\ x\in Z_{J}^{\left(n_{j}\right)}\biggr)\\
 & \leq\lambda\left(I_{J}^{\left(n_{j}\right)}\right)\left(\#\left\{ 0\leq l<h_{J}^{\left(n_{j}\right)}-n:\ \left|\varphi_{n+l}\left(x_{J}\right)-e_{n_{j}}\right|>A\text{ or \ensuremath{\left|\varphi_{l}\left(x_{J}\right)-e_{n_{j}}\right|>A}}\text{ }\right\} +n\right)\\
 & \leq\frac{1}{h_{J}^{\left(n_{j}\right)}}\#\left\{ 0\leq l<h_{J}^{\left(n_{j}\right)}-n:\ \left|\varphi_{n+l}\left(x_{J}\right)-e_{n_{j}}\right|>A\text{ or \ensuremath{\left|\varphi_{l}\left(x_{J}\right)-e_{n_{j}}\right|>A}}\right\} +\epsilon.
\end{align*}
where the last inequality follows by using that $\lambda\left(I_{J}^{\left(n_{j}\right)}\right)h_{J}^{\left(n_{j}\right)}=\lambda\left(Z_{J}^{\left(n_{j}\right)}\right)\leq1$
and recalling that by choice of $n_{j}$ we have that $n/h_{J}^{\left(n_{j}\right)}<\epsilon$.
Furthermore, by a change of indexes,

\begin{align*}
\frac{1}{h_{J}^{\left(n_{j}\right)}}\#\Biggl\{0\leq l<h_{J}^{\left(n_{j}\right)}-n:\ \left|\varphi_{l}\left(x_{J}\right)-e_{n_{j}}\right|>A & \,\mathrm{or}\,\left|\varphi_{n+l}\left(x_{J}\right)-e_{n_{j}}\right|>A\Biggr\}\\
 & \leq\frac{2}{h_{J}^{\left(n_{j}\right)}}\#\left\{ 0\leq l<h_{J}^{\left(n_{j}\right)}:\ \left|\varphi_{l}\left(x_{J}\right)-e_{n_{j}}\right|>A\text{ }\right\} \\
 & =2\mu_{n_{j}}^{J}\left(\left|\sum_{k=1}^{n_{j}}\xi_{k}\left(X_{k}\right)-e_{n_{j}}\right|>A\right),
\end{align*}
where the last equality follows from Proposition \ref{prop: Connection between temporal sums and MC}.
Therefore, from the relation between the measures $\mu_{n}^{J}$ and
$\mu_{n}$ (see Definition \ref{def: Markov transitions}) it follows
that 
\begin{align*}
\lambda\left(\left|\varphi_{n}\right|>2A\right) & =\sum_{J\in\left\{ L,M,S\right\} }\lambda\left(\left|\varphi_{n}\right|>2A\left|Z_{J}^{\left(n_{j}\right)}\right.\right)\cdot\lambda\left(Z_{J}^{\left(n_{j}\right)}\right)\\
 & \leq3\epsilon+\sum_{J\in\left\{ L,M,S\right\} }2\mu_{n_{j}}^{J}\left(\left|\sum_{k=1}^{n_{j}}\xi_{k}\left(X_{k}^ {}\right)-e_{n_{j}}\right|>A\right)\lambda\left(Z_{J}^{\left(n_{j}\right)}\right)\\
 & =3\epsilon+2\mu_{n_{j}}\left(\left|\sum_{k=1}^{n_{j}}\xi_{k}\left(X_{k}\right)-e_{n_{j}}\right|>A\right).
\end{align*}
It follows from (\ref{eq: Markov inequality-1}) that $\lambda\left(\left\{ x:\left|\varphi_{n}\left(x\right)\right|>2A\right\} \right)<5\epsilon$.
Since $\epsilon$ was chosen arbitrarily, this shows that $\varphi_{n}$
is tight. 
\end{proof}

\subsection{Proof of Theorem \ref{thm: Main thm}.\label{subsec:Proof-of-Theorem}}

We begin this section with a few observations that summarize the results
obtained in the preceding sections in the form that is used in order
to prove Theorem \ref{thm: precise version of main} below from which
the main theorem follows. 

\smallskip{}
Let $0<\alpha<\frac{1}{2}$ be badly approximable and $\beta\in\left(0,1\right)$
be badly approximable with respect to $\alpha$. By Lemma \ref{lem: badly approximable implies Ostrowski bounded}
the pair $\left(\alpha_{0},\beta_{0}\right)$ related to $\left(\alpha,\beta\right)$
via equations (\ref{eq:Correcpondence between alpha and alpha_zero})
and (\ref{eq: Correcspondence between beta and beta_zero}), satisfies
$\left(\alpha_{0},\beta_{0}\right)\in\tilde{X}$. To each such pair,
in Section \ref{subsec: Markov chain towers} we associated a Markov
compactum given by a sequence of transition matrices $\left\{ A_{n}\right\} $
(which are incidence matrices for the substitutions $\left\{ \tau_{n}\right\} $
which describe the Rokhlin tower structure) and Markov measures $\left\{ \mu_{n}\right\} $
with transition matrices $\left\{ p^{\left(n\right)}\right\} $ (defined
in \ref{eq:def mu} and Definition \ref{def: Markov transitions}
respectively). Let$\left\{ X_{k}\right\} $ be the coordinate functions
on the Markov compactum (see \ref{eq: Constructed Random variables})
and $\left\{ \xi_{k}\right\} $ be the functions also defined therein
(see Definition \ref{def: xi functions}), which can be used to study
the behavior of Birkhoff sums of the function $\varphi$defined by
(\ref{eq: Cocycle function}) over $T_{\alpha_{0}}$ in virtue of
as proved in Proposition \ref{prop: Connection between temporal sums and MC}.
We set 
\[
e_{n}:=E_{\mu_{n}}\left(\sum_{k=1}^{n}\xi_{k}\left(X_{k}\right)\right),\qquad\sigma_{n}:=\sqrt{Var_{\mu_{n}}\left(\sum_{k=1}^{n}\xi_{k}\left(X_{k}\right)\right)},
\]
where the subscript $\mu_{n}$ in $E_{\mu_{n}}$ and $Var_{\mu_{n}}$
mean that all integrals are taken with respect to the measure $\mu_{n}$. 

Since the function $\varphi$ defined by (\ref{eq: Cocycle function})
is not a coboundary (see Remark \ref{rem: Infiniteness of Ostrowsky expansion}),
Corollary \ref{cor: var goest to infty} implies that $\sigma_{n}\rightarrow\infty$.
By definition of $\xi_{k}$, combining the assumption that $\alpha$
is badly approximable with the inequality (\ref{eq: Denjoy Koksma}),
we obtain that 
\[
\sup\left\{ \xi^{k}\left(J,j\right):\quad k\in\bbN,\ \left(J,j\right)\in\mathcal{S}_{k}\right\} <\infty.
\]

Finally, for any $n\in\bbN$, set $\xi_{k}^{\left(n\right)}:=\xi_{k}$,
$X_{k}^{\left(n\right)}:=X_{n}$, for $k=n,...,1$. Let us then define
a Markov array $\left\{ X_{k}^{\left(n\right)}:\ n\in\bbN,\ k=n,...,1\right\} $,
where $Prob\left(\left(X_{1}^{\left(n\right)},...,X_{n}^{\left(n\right)}\right)\in A\right)=\mu_{n}\left(A\right)$
for every set $A$ in the Borel $\sigma$-algebra of the space $\Sigma_{n}$.
The observations above together with Corollary \ref{cor: Positivity of contraction coefficient for bbd type}
show that all assumptions of Corollary \ref{cor:Under-the-conditions}
hold for this array. Thus 
\begin{equation}
\lim_{n\rightarrow\infty}\mu_{n}\left\{ \frac{\sum_{k=1}^{N}\xi_{k}\left(X_{k}\right)-e_{N}}{\sigma_{N}}\in\left[a,b\right]\right\} =\frac{1}{\sqrt{2\pi}}\int_{a}^{b}e^{-\frac{x^{2}}{2}}dx.\label{eq: CLT}
\end{equation}
 Moreover, by Proposition \ref{prop: Same order of variances and expectations}
(and the fact the $\sigma_{n}\rightarrow\infty$), (\ref{eq: CLT})
holds with $\mu_{n}$ replaced by $\mu_{n}^{J}$, for any $J\in\left\{ L,M,S\right\} $
(where $\mu_{n}^{J}$ are the conditional measures defined by (\ref{eq:def mu J})). 

\smallskip{}
We can now deduce the temporal CLT for Birkhoff sums. Fix $x\in\left[-1,\alpha_{0}\right)$.
Let us first define the centralizing and normalizing constants for
the Birkhoff sums $\varphi_{n}(x).$ For $n\in\bbN$, let $N=N\left(n\right):=\min\left\{ k:\ n\leq h_{S}^{\left(k\right)}\right\} $.
Let $Z_{J}^{\left(N\right)}$ be the tower at stage $N$ of the renormalization
which contains the point $x$ and let $l_{n}$ be the level of the
tower $Z_{J}^{\left(N\right)}$ which contains $x$, i.e. $l_{n}$
satisfies $x\in T^{l_{n}}\left(I_{J}^{\left(N\right)}\right)$ . Set
$c_{n}\left(x\right):=\varphi_{l_{n}}\left(x'\right)$ where $x'$
is any point in $I_{J}^{\left(N\right)}$, i.e. $c_{n}\left(x\right)$
is the Birkhoff sum over the tower $Z_{J}^{\left(N\right)}$ from
the bottom of the tower and up to the level that contains $x$.

We will prove the following temporal DLT, from which Theorem \ref{thm: Main thm}
follows immediately recalling the correspondence between $R_{\alpha}$
and $T_{\alpha_{0}}$ and the functions $f_{\beta}$ and $\varphi$
(refer to the beginning of Section (\ref{subsec:Continued-fraction-renormalization})).
\begin{thm}
\label{thm: precise version of main}For any $a<b$, 
\[
\lim_{n\rightarrow\infty}\frac{1}{n}\#\left\{ 0\leq k\leq n-1:\ \frac{\varphi_{k}\left(x\right)-c_{n}\left(x\right)-e_{N\left(n\right)}}{\sigma_{N\left(n\right)}}\in\left[a,b\right]\right\} =\frac{1}{\sqrt{2\pi}}\int_{a}^{b}e^{-\frac{x^{2}}{2}}dx.
\]
\end{thm}
The above formulation, in particular, shows that the centralizing
constants depend on the point $x$ and have a very clear dynamical
meaning. The proof of this Theorem, which will take the rest of the
section, is based on a quite standard decomposition of a Birkhoff
sums into special Birkhoff sums. For each intermediate Birkhoff sum
along a tower, we then exploit the connection with the Markov chain
given by Proposition \ref{prop: Connection between temporal sums and MC}
and the convergence given by (\ref{eq: CLT}).
\begin{proof}
Fix $0<\epsilon<1$, $a,b\in\bbR$, $a<b$ and let $n\in\bbN$. By
definition of $N=N\left(n\right)$, the points $\left\{ x,...,T^{n-1}x\right\} $
are contained in at most two towers obtained at the $N^{th}$ level
of the renormalization. Let $K$ be defined by $K:=K\left(n\right)=\max\left\{ k:\ h_{L}^{\left(k\right)}\leq\epsilon n\right\} $.
Evidently, $K\leq N$, and by (\ref{eq: Heights of towers growth})
there exists $C>0$ which depends on $\epsilon$ but not on $n$,
such that $N-K\leq C$. 

Thus, since towers of level $N$ are decomposed into towers of level
$K,$ we can decompose the orbit$\left\{ x,...,T^{n-1}x\right\} $
into blocks which are each contained in a tower of level $K$. More
precisely, as shown in Figure (\ref{fig:Decomposition}), there exist
$0=k_{0}\leq k_{1}<...,<k_{t}\leq n$ and towers $\left(Z_{J_{k_{i}}}^{\left(K\right)}\right)_{i=0}^{t}$
appearing at the $K^{th}$ stage of renormalization, such that $\left\{ T^{k_{i}}x,...,T^{k_{i+1}-1}x\right\} \subseteq Z_{J_{k_{i}}}^{\left(K\right)}$
for $i=0,...,t$. Moreover, for $i=1,...,t-1$, the set $\left\{ T^{k_{i}}x,...,T^{k_{i+1}-1}x\right\} $
contains exactly $h_{J_{k_{i}}}^{\left(K\right)}$ points, i.e. $k_{i+1}-k_{i}=h_{J_{k_{i}}}^{\left(K\right)}$
and the points $T^{k_{i}+j}$, $j=0,...,k_{i+1}-1$ belong to the
$j+1$ level of the tower $Z_{J_{k_{i}}}^{\left(K\right)}$. Since
the orbit segment is contained in at most two towers of level $N$
and each tower of level $N$ contains at most $h_{L}^{\left(N\right)}/h_{S}^{\left(K\right)}$
towers of level $K$, we have that $t=t\left(n\right)\leq2h_{L}^{\left(N\right)}/h_{S}^{\left(K\right)}$
and hence is uniformly bounded in $n$. 

It follows from this decomposition that, for any interval $I\subset\bbR$,

\begin{align}
\frac{1}{n}\#\left\{ 0\leq k\leq n-1:\ \varphi_{k}\left(x\right)\in I\right\}  & =\frac{1}{n}\sum_{i=0}^{t-1}\#\left\{ k_{i}\leq k<k_{i+1}:\ \varphi_{k}\left(x\right)\in I\right\} \label{eq: Upper ineq}\\
 & \leq\frac{1}{n}\sum_{i=1}^{t-1}\#\left\{ k_{i}\leq k<k_{i+1}:\ \varphi_{k}\left(x\right)\in I\right\} +2\epsilon,\nonumber 
\end{align}
where the last inequality follows from the fact that $h_{J_{k_{0}}}^{\left(K\right)}$
and $h_{J_{k_{t}}}^{\left(K\right)}$ are both not greater than $n\epsilon$.
Evidently, we also have the opposite inequality 
\begin{equation}
\frac{1}{n}\#\left\{ 0\leq k\leq n-1:\ \varphi_{k}\left(x\right)\in I\right\} \geq\frac{1}{n}\sum_{i=1}^{t-1}\#\left\{ k_{i}\leq k<k_{i+1}:\ \varphi_{k}\left(x\right)\in I\right\} .\label{eq: Lower Ineq}
\end{equation}
For $i=1,...,t-1$, and $k_{i}<k\leq k_{i+1}$, write 
\begin{align*}
\varphi_{k}\left(x\right) & =\varphi_{k}\left(x\right)-\varphi_{k_{i}}\left(x\right)+\varphi_{k_{i}}\left(x\right)=\varphi_{k-k_{i}}\left(x'\right)+\varphi_{k_{i}}\left(x\right)
\end{align*}
 where $x'$ is any point in $I_{J_{k_{i}}}^{\left(K\right)}$. 
\begin{figure}[!tbh]
\includegraphics[width=0.8\textwidth]{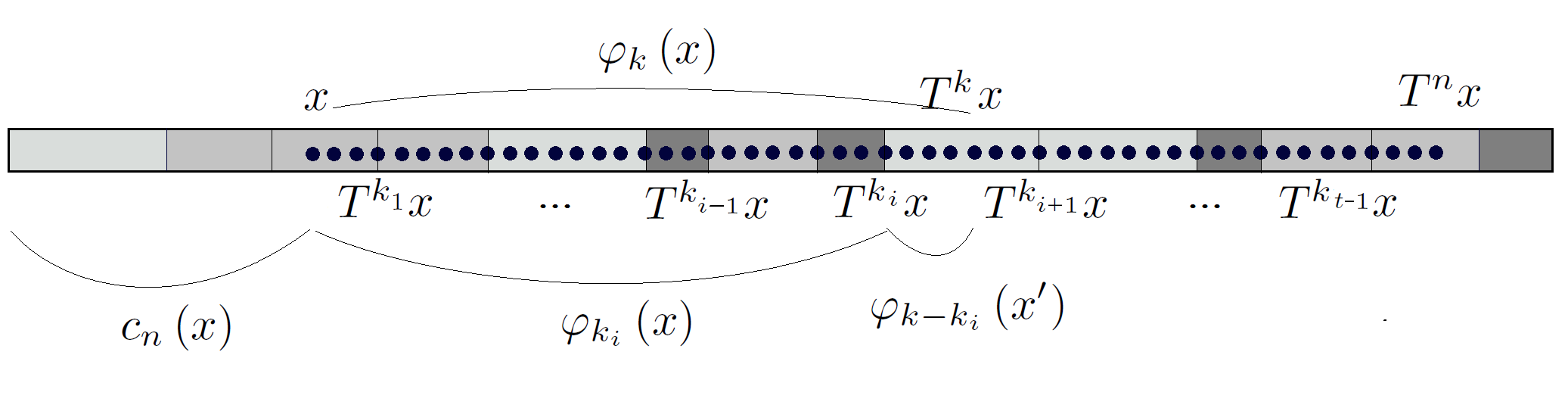}

\caption{\label{fig:Decomposition}Decomposition of the orbit segment $\left\{ x,...,T^{n-1}x\right\} $
into Birkhoff sums along towers of level $N-K$. }

\end{figure}

By definition of $c_{n}\left(x\right)$, $\varphi_{k_{i}}\left(x\right)+c_{n}\left(x\right)=\varphi_{k_{i}}\left(x_{0}\right)$
where $x_{0}$ belongs to the base $I^{\left(N\right)}$ (see Figure
\ref{fig:Decomposition}), thus $\varphi_{k_{i}}\left(x\right)+c_{n}\left(x\right)$
is a sum of special Birkhoff sums over subtowers of $Z_{J}^{\left(K\right)}$,
$J\in\left\{ L,M,S\right\} $. Hence, 

\[
\left|\varphi_{k_{i}}\left(x\right)+c_{n}\left(x\right)\right|\leq\text{\ensuremath{\left(h_{L}^{\left(N\right)}/h_{S}^{\left(K\right)}\right)}}\sup_{J}\left|\varphi_{J}^{(K)}\right|,
\]
by (\ref{eq: Denjoy Koksma}), there exists a constant $\tilde{C}:=\tilde{C}\left(\epsilon\right)$
which does not depend on $n$, such that $\left|\varphi_{k_{i}}\left(x\right)+c_{n}\left(x\right)\right|\leq\tilde{C}$.
It follows from Proposition \ref{prop: Connection between temporal sums and MC}
that 
\begin{align*}
\frac{\#\left\{ k_{i}\leq k<k_{i+1}:\ \frac{\varphi_{k}\left(x\right)-e_{N}-c_{n}\left(x\right)}{\sigma_{N}}\in\left[a,b\right]\right\} }{h_{J_{k_{i}}}^{\left(K\right)}} & =\frac{\#\left\{ 0\leq k<h_{J_{k_{i}}}^{\left(K\right)}:\ \frac{\varphi_{k}\left(x'\right)-\varphi_{k_{i}}\left(x\right)-e_{N}-c_{n}\left(x\right)}{\sigma_{N}}\in\left[a,b\right]\right\} }{h_{J_{k_{i}}}^{\left(K\right)}}\\
 & =\mu_{K}^{J_{k_{i}}}\left(\frac{\sum_{k=1}^{K}\xi_{k}\left(X_{k}\right)-\varphi_{k_{i}}\left(x\right)-e_{N}-c_{n}\left(x\right)}{\sigma_{N}}\in\left[a,b\right]\right).
\end{align*}
Since $\left|N-K\right|=\left|N\left(n\right)-K\left(n\right)\right|<C$,
we have that $\sup_{n}\left\{ \left|e_{N}-e_{K}\right|\right\} <\infty$
and $\frac{\sigma_{N}}{\sigma_{K}}\underset{n\rightarrow\infty}{\longrightarrow}1$.
Moreover, since $\left|\varphi_{k_{i}}\left(x\right)+c_{n}\left(x\right)\right|\leq\tilde{C}$,
it follows from (\ref{eq: CLT}), that for any $J\in\left\{ L,M,S\right\} $
\[
\lim_{n\rightarrow\infty}\frac{1}{h_{J}^{\left(K\right)}}\#\left\{ k_{i}\leq k<k_{i+1}:\ \frac{\varphi_{k}\left(x\right)-e_{N}-c_{n}\left(x\right)}{\sigma_{N}}\in\left[a,b\right]\right\} =\frac{1}{\sqrt{2\pi}}\int_{a}^{b}e^{-\frac{x^{2}}{2}}dx.
\]
Let $n_{0}$ be such that for all $n>n_{0}$ and any $J\in\left\{ L,M,S\right\} $,
\begin{equation}
\left|\frac{1}{h_{J}^{\left(K\right)}}\#\left\{ 0\leq k<h_{J}^{\left(k\right)}-1:\ \frac{\varphi_{k}\left(x\right)-e_{N}-c_{n}\left(x\right)}{\sigma_{N}}\in\left[a,b\right]\right\} -\frac{1}{\sqrt{2\pi}}\int_{a}^{b}e^{-\frac{x^{2}}{2}}dx\right|<\epsilon.\label{eq: Normal approx}
\end{equation}
Then if $n>n_{0}$, by (\ref{eq: Upper ineq}) and (\ref{eq: Normal approx}),
recalling that $\sum_{i=1}^{t-1}h_{J_{k_{i}}}\leq n$,
\begin{align*}
\Bigl|\frac{1}{n}\#\Bigl\{0\leq k\leq n-1:\ \frac{\varphi_{k}\left(x\right)-e_{N}-c_{n}\left(x\right)}{\sigma_{N}} & \in\left[a,b\right]\Bigr|\Bigr\}\\
 & \leq\frac{1}{n}\sum_{i=1}^{t-1}\#\left\{ k_{i}\leq k<k_{i+1}:\ \frac{\varphi_{k}\left(x\right)-e_{N}-c_{n}\left(x\right)}{\sigma_{N}}\in\left[a,b\right]\right\} +2\epsilon\\
 & \leq\text{\ensuremath{\frac{1}{n}\sum_{i=1}^{t-1}\text{\ensuremath{h_{J_{k_{i}}}\text{\ensuremath{\left(\text{\ensuremath{\frac{1}{\sqrt{2\pi}}\int_{a}^{b}e^{-\frac{x^{2}}{2}}}dx\ensuremath{\cdot\frac{1}{n}\sum_{i=1}^{t-1}h_{J_{k_{i}}}}+\ensuremath{\epsilon}}\right)}}}}}+2\ensuremath{\epsilon}}\\
 & \leq\frac{1}{\sqrt{2\pi}}\int_{a}^{b}e^{-\frac{x^{2}}{2}}dx+3\epsilon.
\end{align*}
Similarly, by (\ref{eq: Lower Ineq}), if $n>n_{0}$, using this time
that $\sum_{i=1}^{t-1}h_{J_{k_{i}}}\geq n(1-2\epsilon)$, we obtain
the lower bound 
\begin{align*}
\Bigl|\frac{1}{n}\#\Bigl\{0\leq k\leq n-1:\ \frac{\varphi_{k}\left(x\right)-e_{N}-c_{n}\left(x\right)}{\sigma_{N}} & \in\left[a,b\right]\Bigr|\Bigr\}\\
 & \geq\frac{1}{n}\sum_{i=1}^{t-1}\#\left\{ k_{i}\leq k<k_{i+1}:\ \frac{\varphi_{k}\left(x\right)-e_{N}-c_{n}\left(x\right)}{\sigma_{N}}\in\left[a,b\right]\right\} \\
 & \geq\frac{1}{n}\sum_{i=1}^{t-1}h_{J_{k_{i}}}\left(\frac{1}{\sqrt{2\pi}}\int_{a}^{b}e^{-\frac{x^{2}}{2}}dx-\epsilon\right)\\
 & \geq\left(1-2\epsilon\right)\left(\frac{1}{\sqrt{2\pi}}\int_{a}^{b}e^{-\frac{x^{2}}{2}}dx-\epsilon\right).
\end{align*}
This completes the proof. 
\end{proof}

\subsection*{\textcolor{black}{Acknowledgments.}}

\textcolor{black}{We would like to thank Jon Aaronson, Dima Dolgopyat,
Jens Marklof and Omri Sarig for useful discussions and for their interest
in our work. Both authors are supported by the ERC Starting Grant
ChaParDyn. C. U. is also supported by the Leverhulm Trust through
a Leverhulme Prize. The research leading to these results has received
funding from the European Research Council under the European Union
Seventh Framework Programme (FP/2007-2013) / ERC Grant Agreement n.
335989.}

\bibliographystyle{plain}
\bibliography{Bibiliography}

\end{document}